\documentclass[a4paper,11pt]{amsart}%Estilo del documento

\usepackage{bbm}

\usepackage[utf8]{inputenc}

\usepackage
{geometry}

\usepackage{subfiles} %esto es para modularizar el overleaf
\usepackage[absolute]{textpos}

\usepackage{xparse}
\usepackage{xstring}

\usepackage{mathtools}

\usepackage{mathabx}

\usepackage{tikz-cd}

\usepackage{graphicx, subfigure}
\usepackage[export]{adjustbox}

\usepackage{braket} %para definir \set , \Set y que los conjuntos se vean mas lindos

\usepackage[sc]{mathpazo}

\usepackage{enumerate}

\usepackage{hyperref}
\hypersetup{
    colorlinks,
    citecolor=purple,
    filecolor=purple,
    linkcolor=purple,
    urlcolor=purple
}

\usepackage{xcolor}

\usepackage{paralist}

\usepackage{graphicx,wrapfig,lipsum}

\usepackage{multicol}

\definecolor{mypink3}{cmyk}{0, 0.7808, 0.4429, 0.1412} % el mypink3 no es necesario

\usepackage{enumitem}

\usepackage{mathrsfs}

\usepackage{amsthm}

\theoremstyle{plain}
\newtheorem{theorem}{Theorem}[section]
\newtheorem{lemma}[theorem]{Lemma}
\newtheorem{proposition}[theorem]{Proposition}
\newtheorem{proposition/definition}[theorem]{Proposition/Definition}
\newtheorem{corollary}[theorem]{Corollary}

\theoremstyle{definition}
\newtheorem{definition}[theorem]{Definition}

\newtheorem{example}[theorem]{Example}
\newtheorem{obs}[theorem]{Obs}

\theoremstyle{remark}
\newtheorem{remark}[theorem]{Remark}
\newcommand{\boxs}{\overline{\dim}_B}
\newcommand{\boxi}{\underline{\dim}_B}

\newcommand{\udim}{\overline{\dim}_\theta}
\newcommand{\ldim}{\underline{\dim}_\theta}

\newcommand{\udimp}{\overline{\dim}_\phi}
\newcommand{\wig}[1]{\widetilde{#1}}
\newcommand{\dgl}{\underline{\dim}^\Phi}
\newcommand{\dgu}{\overline{\dim}^\Phi}
\newcommand{\de}{\delta}

\newcommand{\te}{\theta}

\newcommand{\norm}[1]{ \lVert {#1} \rVert}
\newcommand{\Norm}[1]{\left \lVert {#1} \right \rVert}

\newcommand{\supp}[1]{\operatorname{supp} \left ( #1 \right )}

\newcommand{\Rn}{\mathbb{R}^n}

\newcommand{\dimlocs}[1]{ \overline{\dim}_{\text{loc} } #1}

\newcommand{\dimloci}[1]{ \underline{\dim}_{\text{loc} } #1}
\newcommand{\dimloc}[1]{\dim_{\text{loc}} #1}

\NewDocumentCommand{\pspace}{m O{\Rn}}{\ensuremath{L^{#1} \left ( #2 \right )}}

\NewDocumentCommand{\pnorm}{m O{p}}{\norm{#1}_{#2}}
\NewDocumentCommand{\pNorm}{m O{p}}{\Norm{#1}_{#2}}

\newcommand\rest[2]{{% we make the whole thing an ordinary symbol
  \left.\kern-\nulldelimiterspace % automatically resize the bar with \right
  #1 % the function
  \vphantom{\big|} % pretend it's a little taller at normal size
  \right|_{#2} % this is the delimiter
  }}
 
\title[Intermediate dimensions of measures]{Intermediate dimensions of measures: \\{\Small Interpolating between Hausdorff and Minkowski dimensions}}
\date{\today}
\author{Nicolas E. Angelini} \address{Universidad Nacional de San Luis \\Facultad de Ciencias Fisico Matematicas y Naturales\\Departamento de Matematica\\ and \\CONICET \\Instituto de Matematica Aplicada\\ San Luis (IMASL)}
\email{nicolas.angelini.2015@gmail.com}
\author{Ursula M. Molter}
\address{Universidad de Buenos Aires\\ Facultad de Ciencias Exactas y Naturales\\ Departamento de Matematica \\and\\ UBA-CONICET\\ Instituto de Investigaciones Matematicas\\ Luis A. Santalo (IMAS)}
\email{umolter@conicet.gov.ar}
\author{Jose M. Tejada}
\address{Universidad de Buenos Aires\\ Facultad de Ciencias Exactas y Naturales\\ Departamento de Matematica \\and\\ UBA-CONICET\\ Instituto de Investigaciones Matematicas\\ Luis A. Santalo (IMAS)}
\email{josemtejada1999@gmail.com}

\begin{document}

\begin{abstract}
In this paper, we define a family of dimensions for Borel measures that lie between the Hausdorff and Minkowski dimensions for measures, analogous to the intermediate dimensions of sets. Previously, Hare et.~al.~in \cite{HARE_2020} defined families of dimensions that interpolate between the Minkowski and Assouad dimensions for measures. Additionally, Fraser, in \cite{FraserFourier} introduced an additional family of dimensions that interpolate between the Fourier and Sobolev dimensions of measures. Our results address a "gap" in the study of dimension interpolation for measures, almost completing the spectrum of intermediate dimensions for measures: from Hausdorff to Assouad dimensions.

Furthermore, Theorem \ref{teoprincipal} can be interpreted as a "reverse Frostman" lemma for intermediate dimensions. We also obtain a capacity-theoretic definition that enables us to estimate the intermediate dimensions of pushforward measures by projections.
\end{abstract}

\maketitle



\section{Introduction}

In the last few years, dimension interpolation techniques have been developed with the goal of understanding deeper relations between pairs of dimensions. These ideas were developed for sets in \cite{falconer2020intermediate}, \cite{survey}, \cite{Banaji}, and \cite{HAREsets}, among others. However, for the classic set dimensions that these papers study there exist also the notions of dimensions for measures.
%In particular, \cite{Falconerminkowski} defines Minkowski dimensions for measures. 
Because of this, it is natural to think about defining intermediate dimensions for measures. 
%such that they behave similar to the Minkowski dimension (Mainly because intermediate dimensions for sets are in many aspects similar to the Minkowski dimension).

%mencionar explicitamente los nombres de losautores
Interpolation between different pairs of dimensions for measures can be found in 
\cite{FraserFourier} (between the Fourier and Sobolev dimensions) and \cite{HARE_2020} (between the Minkowski and Assouad dimensions).

In our case, we will interpolate between the upper Hausdorff and the Minkowski dimension of measures (Minkowski dimensions for measures were introduced in \cite{Falconerminkowski}. For precise definitions of these dimensions see section~\ref{defs}).  Given that our dimensions lie between the upper Hausdorff and Minkowski dimensions, our notion almost fills the 'gap' that these other families of dimensions do not cover. 

Interestingly, many theorems for sets can be rewritten for measures, and as we will see in this paper, their set version can sometimes be deduced directly from the corresponding measure version, using a result similar to Theorem~2.1 from \cite{Falconerminkowski}. 

Our main theorem can also be seen as a sort of 'inverse Frostman lemma' for intermediate dimensions.

Let $\dim_H (\mu)$ and $ \boxs (\mu)$ denote the Hausdorff and Box dimensions, respectively, where $\boxs (\mu)$ is defined in \cite{Falconerminkowski}. The definition of intermediate dimensions that we give is based on four properties that we would like them to satisfy. 

\begin{enumerate}
    \item {$\udim (\mu) \leq \overline{\dim}_\phi (\mu)$ for all $0 \leq \theta \leq \phi \leq 1$.}
    \item {$\udim (\mu) \leq \boxs (\mu)$.}
    \item {$\dim_H (\mu)\leq \udim (\mu)$.}
    \item {For every compact set $E$, \\ $\udim(E) = \min \{\udim (\mu) : \mu \text{ finite and } \supp\mu = E \}$}.
\end{enumerate}

%As a final observation we simply point out that we will write  $B(x, r)$ for closed balls.

In Sections \ref{AAAB}, we provide an alternative definition of the intermediate dimension of a measure in terms of capacities associated with certain kernels. Building on this, we show that the intermediate dimensions of the image of a measure \( \mu \) under the orthogonal projection onto almost all \( m \)-dimensional subspaces depend only on \( m \) and \( \mu \). In other words, they are almost surely independent of the choice of subspace. Finally, in section \ref{section6} we state an additional result regarding projections of measures, and the behavior of the dimensions of projections when $\te\rightarrow 0$.

\section{Definitions}\label{defs}

Unless specified, the measures that we study will always be finite, Borel regular (in $\Rn$ for some $n$) with compact support. However by  Remark~\ref{obssopcomp}, requiring compact support is not really a restriction. By $B(x ,r)$ we mean the closed balls in $\Rn$, and by $\mathbbm{1}_{A}$, the indicator function of $A$.

Before defining the intermediate dimensions, we recall some preliminary concepts.

\begin{definition}
    the \textbf{upper and lower local dimensions of $\mu$ at $x \in \Rn$} are defined as:
    $$\dimlocs{\mu} (x) := \limsup_{r \to 0} \frac{\log \mu (B(x,r))}{\log r} \:\:\:\:,\:\:\:\: \dimloci{\mu} (x) := \liminf_{r \to 0} \frac{\log \mu (B(x,r))}{\log r}.$$
    In case they are equal, their common value is written as $\dimloc \mu (x)$.
\end{definition}

It is easy to see that for $x \notin \supp{\mu}$, $\dimloc{\mu}(x) = \infty$, and if $x$ is an atom of $\mu$, then $\dimloc{\mu}(x)=0$. More information about these two dimensions can be found in \cite{Falconertechniques}. 

\begin{definition}

The \textbf{upper Hausdorff dimension} of $\mu$ is defined by:

\begin{align*}
    \dim_H^* (\mu) &=\inf \{ s: \dimloci{\mu} (x) \leq s \: \text{ for } \mu-\text{almost every } x \}. \\
    & = \inf \{ \dim_H(E) : E \text{ Borel } / \: \mu(\Rn \setminus E)=0 \},
\end{align*}

and the \textbf{Hausdorff dimension} of $\mu$ is defined by:

\begin{align*}
    \dim_H (\mu) &=\sup \{ s: \dimloci{\mu} (x) \geq s \: \text{ for } \mu-\text{almost every } x \}. \\
    & = \inf \{ \dim_H(E) : E \text{ Borel } / \: \mu(E)>0 \}.
\end{align*}

\end{definition}

Note that,  $\dim_H (\mu) \leq \dim_H^* (\mu)$ and the inequality can be strict (see \cite[Prop. 10.12]{Falconertechniques}).

\begin{definition} \label{defdimmink}
    Following \cite{Falconerminkowski}, we define the \textbf{Upper and lower Minkowski dimensions} for measures by:
    \medskip

        $ \boxs (\mu) := \inf \Set{  s \geq 0 : 
        \begin{array}{c}
        \: \exists \: c>0 \text{ such that, for all } \:  x \in \supp{\mu} \text{ and } r \in (0, 1):
        \\*[1mm]
        \: \mu (B(x, r)) \geq c \cdot r^s 
        \end{array}  } $.
        
        \medskip
        
        $ \boxi (\mu) := \inf \Set{  s \geq 0 : 
        \begin{array}{c}
        \: \exists c>0 \text{ and } (r_n)_{\mathbb{N}} \subseteq \mathbb{R}_{\geq0} \text{ with } r_n \to 0, \text{ such that:  } \\*[1mm]
        \: \mu (B(x, r_n)) \geq c \cdot r_{n}^s \: \forall x \in \supp{\mu} \text{ and } n \in \mathbb{N}
        \end{array}  } $.
        
   % \end{itemize}
\end{definition}

\begin{remark} \label{obssopcomp}

    For measures with unbounded support, $\boxs(\mu) = \boxi(\mu) = \infty$, because one can take some fixed $r>0$, and a sequence $x_n \rightarrow \infty$, such that  $B(x_n, r)$ has positive measure, and $B(x_m, r) \cap B (x_n, r) = \emptyset$. As the measure is finite, the measures of such balls must decrease to $0$, and so no $s>0$ would satisfy the desired inequalities in the previous definition. Consequently, one does not lose much considering only compactly supported measures.
\end{remark} 

In order to motivate our definitions, the following observation can be established first:

\begin{lemma} \label{obsminkdelta} For every $\mu$,
$$\boxs (\mu) = \inf \{  s \geq 0 : \exists c>0 \text{ and }\delta_0 >0 / \mu (B(x, \delta)) \geq c \cdot \delta^s \: \forall x \in \supp{\mu} \text{ and } \delta< \delta_0  \}.$$
    
\end{lemma}

\begin{proof} 

Clearly,
$$\boxs (\mu) \geq \inf \{  s \geq 0 : \exists c>0 \text{ and }\delta_0 >0 / \mu (B(x, \delta)) \geq c \cdot \delta^s \: \forall x \in \supp{\mu} \text{ and } \delta< \delta_0  \}.$$
To prove the non-trivial inequality, suppose that $s$ satisfies the right hand side of the equality. %, that is, there exist $c>0,\: \delta_0$ such that $\mu (B(x, \delta)) \geq c \cdot \delta^s$ for $x \in \supp{\mu}$ and $\delta \leq \delta_0 $. 
Let $ r \in (\delta_0, 1)$ and $ x \in \supp{\mu}$. Then, 

$$ \delta_0^s \cdot c \cdot r^s \leq \delta_0 ^s \cdot c \leq \mu (B(x, \delta_0)) \leq \mu (B(x, r)).$$

Because of this, choosing $\tilde c = \min(c, c \delta_0^s)$ we have $s\geq \boxs(\mu)$. 
\end{proof}

Before proceeding with intermediate dimensions, we need one final lemma, which will be necessary later. The next definition and lemma, along with a discussion of these topics, can be found in \cite{Banaji}.

\begin{definition}
Let $y \in (0, \infty)$. A function $\Phi: A \subseteq \mathbb{R} \rightarrow \mathbb{R}$ with $A$ containing an interval $(0, y)$ is said to be \textbf{y-admissible} if:
\begin{itemize}
    \item $0 < \Phi(\delta)\leq \delta \: $for all $ \delta \in (0, y).$
    \item $\displaystyle{\frac{\Phi(\delta)}{ \delta}} {\rightarrow} 0$ when $\delta \rightarrow 0^+.$
    \item $\Phi$ is non-decreasing in $(0, y)$.
\end{itemize}
$\Phi$ is called \textbf{admissible} if the exists $y$ such that $\Phi$ is $y$-admissible.
\end{definition}

\begin{remark} \label{obsthetanot01}
For all $\theta \in (0, 1)$, excluding $\theta = 0\: \text{and}\:  1$ the function $\Phi(\delta) = \delta^{1/\theta}$ is admissible. %Notice how the case $\theta = 1$ is not.
\end{remark}

\begin{definition} \label{defdimintgener}
Let $\Phi$ be admissible, and $\emptyset \neq E \subseteq \Rn$ a totally bounded set. The \textbf{generalised intermediate dimensions} of $E$ are defined as:
\medskip

     $\dgu (E) := \inf \Set{s \geq 0 |
    \begin{array}{c}
    \forall \varepsilon > 0 \: \exists\: \delta_0> 0 / \forall \: 0<\delta< \delta_0 \: \exists \left \{ U_i \right \} \text{covering of } E \\*[1mm]
    \text{ such that } \Phi(\delta) \leq |U_i| \leq \delta \: \forall i \: , \: \sum |U_i|^s < \varepsilon 
    \end{array} }.$

\medskip

    $\dgl(E) := \inf \Set{s \geq 0 |
    \begin{array}{c}
    \forall \varepsilon > 0,\delta_0 > 0 \: \exists \: \delta \in (0, \delta_0) \text{ and } \left \{ U_i \right \} \text{ a covering of } E \\*[1mm]
    \text{ such that } \Phi(\delta) \leq |U_i| \leq \delta \: \forall i \: , \: \sum |U_i|^s < \varepsilon
     \end{array}  }.$
\vspace{1mm}

If $\dgl (E) = \dgu(E)$, we write $\dim^\Phi (E)$. 
\end{definition}
    
When the function $\Phi(\delta)$ is taken to be $\Phi(\delta) = \delta^{1/\theta}$, one obtains the $\theta$-intermediate dimensions, as defined in \cite{survey}. Even though the function $\Phi = 0$ is not admissible, one can still use it to recover the Hausdorff dimension of the set $E$ (which is not a generalised intermediate dimension, as one can prove that $\dgu (E) = \dgu (\overline{E})$). It is also important to note that the Minkowski dimension can be understood as a generalised intermediate dimension, choosing an appropriate $\Phi$ (which is not $\Phi(\delta) := \delta$ since it is not admissible). 

We will need the following lemma later.

\begin{lemma} \label{observacióndesigualdadtrivial}
Let $\Phi$ be admissible and $E \subseteq \Rn$ bounded. Suppose there exist $C ,\: \de_0>0 $, such that for every $\de < \de_0$ one can find a covering $\set{U_i}$ of $E$ such that $\Phi(\delta) \leq |U_i| \leq \delta$ and $\sum |U_i|^s \leq C$. Then, $\dgu(E)\leq s$.
\end{lemma}
\begin{proof} 

In order to prove the desired inequality it is enough to conclude that for each $ \eta >0$, $\dgu (E)< s + \eta$. For this, let $\varepsilon>0$, and take:
$$\de_1 = \min\set{ \de_0, (\varepsilon/c)^{1/\eta}}.$$

Now, for all $\de < \de_1 < \delta_0$, there exists $\set{U_i}$ a covering of $E$ with $\Phi(\delta) \leq |U_i| \leq \delta$ and $\sum |U_i|^s \leq C$. Then, we have:

$$\sum |U_i|^{s + \eta} = \sum |U_i|^s |U_i|^{\eta} \leq \sum |U_i|^s\de^\eta \leq C \de^\eta < C \de_1^\eta < \varepsilon.$$

Therefore,  $\dgu(E) \leq s + \eta \: $ for all $ \: \eta >0$, and $\dgu(E) \leq s$.
\end{proof}

\begin{remark}

Lemma \ref{observacióndesigualdadtrivial} remains valid for $\dim_H$ and $\boxs$, even though $\Phi = 0$ and $\Phi (\delta) = \delta$ are not admissible.

\end{remark}

\begin{remark} \label{condsufidimint}
One can formulate a version of the lemma for lower dimensions:

Suppose that there exists $ \: c > 0$ such that for every $\de_0$ there exists $\de < \de_0$ and a covering $\set{U_i}$ with $\sum|U_i|^s\leq c$ and $\Phi(\de) \leq |U_i|\leq \de$. Then, $\dgl(E) \leq s$.

%\begin{proof}
%It is enough to see that for every $\eta > 0$, one has the inequality $s+\eta > \dgu(E)$. Let $\varepsilon, \de_0 >0$. Then, by hypothesis, there $ \exists \: \de<\de_0$ such that $c\cdot \de^\eta < \varepsilon$ and such that there exists a cover $\set{U_i}$ with $\Phi(\de) \leq |U_i|\leq \de$ and $\sum|U_i|^s\leq c$. Then:

%$$\sum |U_i|^{s+\eta} = \sum |U_i|^s|U_i|^\eta \leq c \cdot \de^\eta<\varepsilon.$$

%\end{proof}

\end{remark}

\section{Intermediate dimensions for measures}

Interestingly, given that intermediate dimensions of sets are defined by imposing restrictions on the diameters of sets belonging to certain coverings, one can define these concepts for measures by taking restrictions on the diameter of balls centered on points in the support of the measure. To be precise: 

\begin{definition} \label{defdimintmed}
    Let $\theta \in [0, 1]$. The \textbf{$\theta$-intermediate dimensions} of $\mu$ are defined by:
    
    \medskip
    
%         $\udim (\mu) := \inf \Set{s \geq 0:
%        \begin{array}{c}
%         \: \exists \: c>0 \text{ and } \: \delta_0 > 0 / \forall \: \delta < \delta_0 \text{ and } x \in \supp{\mu},\:\\*[1mm]
%         \exists \: r \in [\delta^{1/\theta}, \delta] \text{ with } \: \mu (B(x, r)) \geq c \cdot r^s
%         \end{array} }.$

        $\udim (\mu) := \inf \Set{s \geq 0:
        \begin{array}{c}
         \: \exists \: c>0 \text{ and } \: \delta_0 > 0 / \text{ for } \mu-\text{a.e. } x \in \supp{\mu} \text{ we have: }\\*[1mm]
         \forall \: \delta<\delta_0  \: \: \exists \: r \in [\delta^{1/\theta}, \delta] \text{ with } \: \mu (B(x, r)) \geq c \cdot r^s
         \end{array} }.$

\medskip

%         $\ldim (\mu) := \inf \Set{s \geq 0:
%        \begin{array}{c}
%         \: \exists \: c>0 \text{ and } \: \delta_n \to 0 / \forall \: n \in \mathbb{N} \text{ and } x \in \supp{\mu},\:\\*[1mm]
%         \exists \: r \in [\delta_n^{1/\theta}, \delta_n] \text{ with } \: \mu (B(x, r)) \geq c \cdot r^s
%         \end{array} }.$

         $\ldim (\mu) := \inf \Set{s \geq 0:
        \begin{array}{c}
         \: \exists \: c>0 \text{ and } \: \delta_n \to 0 / \text{ for } \mu-\text{a.e. } x \in \supp{\mu},\: \text{ we have: } \\*[1mm]
        \forall \: n \in \mathbb{N} \: \exists \: r \in [\delta_n^{1/\theta}, \delta_n] \text{ with } \: \mu (B(x, r)) \geq c \cdot r^s
         \end{array} }.$

\medskip

For the case $\theta=0$, we take $\de^{1/\theta}=0$ in the above definition. Also, notice that the $r$ in the above definition depends on the element $x$ - it is not uniform.

In Section \ref{defgeneral} we will briefly discuss an analogue of the generalised intermediate dimensions for measures. It is also clear that we can replace the intervals $[\delta^{1/\theta}, \delta]$ with $[\delta, \delta^\theta]$ when $\theta >0$.

\end{definition}

\begin{remark}
Clearly, these dimensions satisfy the next inequalities:
\begin{enumerate}
    \item $\udim (\mu) \leq \overline{\dim}_\phi$ for all  $0 \leq \theta \leq \phi \leq 1$.
    \item $\udim (\mu) \leq \boxs (\mu)$.
    \item $\overline{\dim}_0 (\mu)=\underline{\dim}_0 (\mu)=\dim_H^* (\mu).$
\end{enumerate}
\end{remark}

It is not very clear that $\overline{\dim}_1 (\mu) = \boxs (\mu)$, given that in the definition of intermediate dimensions, we require certain conditions to be true for almost all $x$. However, as we shall see, we will be able to replace ``${\rm almost\ all}\: x$" with ``${\rm for\ all} \: x$".

\begin{lemma}\label{elctpesunparatodo}

Let $\theta \in (0,1]$. Then:

    \medskip
    
         $\udim (\mu) := \inf \Set{s \geq 0:
        \begin{array}{c}
         \: \exists \: c>0 \text{ and } \: \delta_0 > 0 / \forall \: \delta < \delta_0 \text{ and } x \in \supp{\mu},\:\\*[1mm]
         \exists \: r \in [\delta^{1/\theta}, \delta] \text{ with } \: \mu (B(x, r)) \geq c \cdot r^s
         \end{array} }.$
 
\medskip

$\ldim (\mu) := \inf \Set{s \geq 0:
\begin{array}{c}
\: \exists \: c>0 \text{ and } \: \delta_n \to 0 / \forall \: n \in \mathbb{N} \text{ and } x \in \supp{\mu},\:\\*[1mm]
\exists \: r \in [\delta_n^{1/\theta}, \delta_n] \text{ with } \: \mu (B(x, r)) \geq c \cdot r^s
\end{array} }.$
\end{lemma}

\begin{proof}

The inequality $\geq$ is clear. For the nontrivial inequality, suppose that $s>\udim \mu$, so that there exist $c,\ \delta_0>0$ and $A\subseteq\supp\mu$ with $\mu(\supp\mu \setminus A)=0$ such that, for all $x \in A$, we have, that for every $\delta<\delta_0$ there exists $ r \in [\delta^{1/\theta}, \delta]$ such that $\mu (B(x, r)) \geq c \cdot r^s.$

We need to verify that the same remains valid for every $x\in\supp\mu$. Let $x\in\supp\mu$, $\delta<\delta_0$ and $\wig{\delta} := \delta/2$. Then, take $y\in A$ such that $d(x, y)\leq \wig{\delta}^{1/\theta}$. %(the existence of such $y$ can be deduced by the fact that if $\mu({x})>0$ then $x \in A$ by definition and if $\mu({x})=0$, since $x\in \supp\mu$, every ball $B(x, r)$ has positive measure).

By hypothesis, the exists $\wig{r}\in  \left [ \wig{\delta}^{1/\theta}, \wig{\delta} \right ]$ such that:
$$\mu(B(y, r))\geq c \cdot \wig{r}^s.$$
Then, taking $r:= 2\wig{r}$, we have that $B(y, \wig{r})\subseteq B(x, r)$, and so
$$\mu(B(x, r))\geq \mu(B(y, \wig{r})) \geq c\cdot \left ( \frac{r}{2} \right )^s = \frac{c}{2^s}r^s.$$

It is clear that $r \in \left [2  \wig{\delta}^{1/\theta}, 2 \wig{\delta} \right ]= \left [ \frac{2}{2^{1/\theta}}\delta^{1/\theta}, \delta \right ]$. However, we need $ r \in \left [ \delta^{1/\theta}, \delta \right ]$. If
 $r\geq \delta^{1/\theta}$ there is nothing to do.
    
    If $r< \delta^{1/\theta}$ then 
$$\mu(B(x, \delta^{1/\theta}))\geq \mu(B(x, r))\geq \frac{c}{2^s} \cdot r^s \geq \frac{c}{2^s} \cdot \left ( \frac{2}{2^{1/\theta}} \right )^s \left ( \delta^{1/\theta} \right )^s = c \cdot \frac{1}{2^{s/\theta}} \left ( \delta^{1/\theta} \right )^s.$$

Therefore, changing the constant $c$, we have that $x$ satisfies the same condition as $y$ and so the right-hand side in the inequality is smaller than $s$.
\end{proof}

Keeping this result in mind, it is clear that $\overline{\dim}_1 (\mu) = \boxs (\mu)\: , \: \underline{\dim}_1 (\mu) = \boxi (\mu)$, and we can conclude that the intermediate dimensions interpolate between the (upper) Hausdorff and Minkowski dimensions. Whenever $\theta>0$, we will work with the right-hand expressions in this lemma instead of our original definition, because it is a little easier to work with. 

\begin{remark}\label{nopara0}
    For $\theta=0$, lemma \ref{elctpesunparatodo} is false. In fact, one can prove without much difficulty that, if $\overline{\dim}_0' (\mu),  \underline{\dim}_0' (\mu)$ are defined by taking $\theta=0$ in the right-hand expressions given in lemma \ref{elctpesunparatodo}, then:
  \begin{align*}
    \overline{\dim}_0 (\mu) &= \inf \{  s : \: \forall \: x \in \supp{\mu}, \dimloci{\mu} (x) \leq s \} \\
    & = \inf \{  s: \exists c >0 / \: \forall \: x \in \supp{\mu} 
 \: \exists \: r_n \to 0 \text{ such that } \mu (B (x, r_n)) \geq c  r_n^s  \} \\
    &  = \underline{\dim}_0 (\mu).
\end{align*}
\end{remark}

%\begin{remark} \label{obstrivial}
%Clearly, these dimensions satisfy the next inequalities:
%\begin{enumerate}
%    \item $\udim (\mu) \leq \overline{\dim}_\phi$ for all  $0 \leq \theta \leq \phi \leq 1$.
%    \item $\udim (\mu) \leq \boxs (\mu)$.
%    \item $\overline{\dim}_1 (\mu) = \boxs (\mu), \text{ and } \underline{\dim}_1 (\mu) = \boxi (\mu)$, the first one because of \ref{obsminkdelta}.
%\end{enumerate}

%In other words, they possess the first two of the desired properties mentioned in the introduction. Also, when $\theta = 0$, one can prove the following, without much difficulty:

%\begin{align*}
%    \overline{\dim}_0 (\mu) &= \inf \{  s : \: \forall \: x \in \supp{\mu}, \dimloci{\mu} (x) \leq s \} \\
%    & = \inf \{  s: \exists c >0 / \: \forall \: x \in \supp{\mu} 
% \: \exists \: r_n \to 0 \text{ such that } \mu (B (x, r_n)) \geq c  r_n^s  \} \\
%    &  = \underline{\dim}_0 (\mu).
%\end{align*}

%\end{remark}

%Clearly, $\overline{\dim}_0 (\mu) $ need not be $\dim_H^*(\mu)$. However, this can be easily fixed:

%One can define $\udim'(\mu)$ by replacing ''$x \in \supp{\mu}$'' with ''$ \mu\text{ -a.e. } x \in \supp{\mu}$''. This slight change does not affect the value of the dimensions for $\theta>0$, that is, $$\udim(\mu) = \udim'(\mu).$$ However, $\underline{\dim}_0'(\mu) = \dim_H^*(\mu)$. In conclusion, our dimensions do interpolate between the upper Hausdorff and the Minkowski, and the difference in $\udim(\mu)$ and $ \udim'(\mu)$ is nothing more than a bothersome detail.

\begin{proposition}\label{musop1}
For every $\theta \in (0, 1]$,
$$\udim (\supp{\mu}) \leq \udim (\mu).$$
\end{proposition}

\begin{proof}
First note that if $(A_i)_{\mathbb{N}}$ are Borel sets such that  $\sum \mathbbm{1}_{A_i} \leq M$ for some $M \in \mathbb{N}$, then, $\sum \mathbbm{1}_{A_i} \leq M \cdot \mathbbm{1}_{ \cup A_i} $ and therefore, taking integrals and using monotone convergence, we have that $\sum \mu(A_i) \leq M \cdot \mu(\bigcup_i A_i).$

%In order to show this, note that $\sum \mathbbm{1}_{A_i} \leq M \cdot \mathbbm{1}_{ \cup A_i} $, so then, taking integrals, we get $\int \sum \mathbbm{1}_{A_i} d\mu \leq M \cdot \mu (\bigcup A_i)$. Then, by monotone convergence the result follows.

To prove the proposition, take $s>\udim (\mu)$. Using Besicovitch's covering theorem, there exist $c, \delta_0$ and $M$ such that for every $\delta<\delta_0$ there exists a cover $\{B_i=B(x_i, r_i)\}_{\mathbb{N}}$ of $\supp{\mu}$, with the following properties:

$$\sum \mathbbm{1}_{B_i}(x) \leq M \: , \:\: \delta^{1/\theta} \leq r_i \leq \delta \: \text{ and } \: \mu(B_i) \geq c \cdot r_i^s.$$

Also, we can take $\delta_0$ such that for $\delta<\delta_0$ we have $\frac{2^{2/\theta}}{4} \leq \delta^{1 - 1/\theta}$. Then, as $\delta^{1/\theta} \leq r_i \leq \delta$, we obtain $ 2 \delta^{1/\theta} \leq | B_i | \leq 2 \delta$, hence the diameters are not in the 'appropriate' range (which would be $ 2^{1/\theta} \delta^{1/\theta} \leq | B_i | \leq 2 \delta$), so we need to 'enlarge' those $B_i$ that don't fall in the correct range. For this we choose a new covering $\wig{B_i}$ as follows:

If $(2 \delta)^{1/\theta} \leq | B_i | \leq 2 \delta$, we leave $B_i$ as it is, taking $\widetilde{B_i}:= B_i$. Otherwise (that is, if $ 2 \delta^{1/\theta} \leq | B_i | \leq (2 \delta)^{1/\theta}$), we take $\widetilde{B_i} := aB_i$ with $a= 2^{1/\theta}/2$, so that we get $(2\delta)^{1/\theta} \leq  \left| aB_i \right|\leq 2\delta$.

We have:

$$\sum | \widetilde{B_i} |^s \leq  \sum a^s \left| B_i \right|^s = \sum (2a)^s r_i^s \leq \frac{(2a)^s}{c} \sum \mu (B_i) \leq \frac{M \cdot (2a)^s}{c} \mu(\Rn)< \infty.$$

Now, because of \ref{observacióndesigualdadtrivial} we conclude that $s \geq \udim (\supp{\mu})$.
\end{proof}

\begin{remark}

This argument will not work for $\te=0$ because of Remark \ref{nopara0}: Right at the beginning of the proof, it is not a priori possible to take a cover $B_i$ of $\supp\mu$ as described in the proof, because in the condition in Definition \ref{defdimintmed}, we cannot replace the ''$\mu-a.e.$'' for a ''for all $x$''.

\end{remark}

%\begin{corollary}
%$\dim_H^*(\mu) \leq \udim (\mu)$.
%\end{corollary}
    
%\begin{proof}
%$\dim_H^*(\mu) \leq \dim_H(\supp{\mu}) \leq \udim (\supp{\mu})\leq \udim (\mu).$
%\end{proof}   

Just like in the case of sets, nontrivial inequalities relating different intermediate dimensions are greatly desired. One such result is the following, which is true for sets (in fact, many results in this paper will prove translations of well-known inequalities for sets):
\begin{proposition} \label{propdesigualdad1}
    For all $0 < \theta \leq \phi \leq 1$,
    $$\udimp (\mu) \leq \frac{\phi}{\theta} \udim (\mu).$$

In particular, if $\udim (\mu) = 0$ for some $\theta \in (0, 1]$ then $\boxs (\mu) = 0$.
\end{proposition}

\begin{proof}

Let $s> \udim (\mu)$, and $c, \delta_0>0$ as in the definition for $\udim (\mu)$. Then, taking $\delta<\delta_0$ and $x \in \supp{\mu}$, we need to prove that there exists $r \in [\delta^{1/\phi}, \delta]$ such that $\mu(B(x, r))\geq c r^{s \cdot \phi/\theta}$.

By definition, it is clear that there exists $r \in [\delta^{1/\theta}, \delta]$ with $\mu(B(x, r))\geq c r^{s}$. We distinguish two cases:

\begin{enumerate}
    \item If $r \in [\delta^{1/\phi}, \delta]$, then the same $r$ works, because $r^s \geq r^{s \cdot \phi/\theta}$.
    \item If $r \in [\delta^{1/\theta}, \delta^{1/\phi}]$, then we can see that $\delta^{1/\phi}$ works:

    $$\mu(B(x, \delta^{1/\phi})) \geq \mu(B(x, r))\geq c \cdot r^s \geq c \cdot (\delta^{1/\phi})^{\frac{\phi}{\theta}s}. $$
\end{enumerate}

In conclusion, for all $\delta, x$ we can find $r \in [\delta^{1/\phi}, \delta]$ with $\mu(B(x, r)) \geq c r^{s \phi/\theta}$, and $\udimp (\mu) \leq s \frac{\phi}{\theta}$.

\end{proof}

\begin{remark}
This result shows in particular that the intermediate dimensions of measures can be infinite: Theorem 4.1(iii) \cite{HARE_2020} gives an example of a measure such that $\boxs\mu = \infty$, and so every intermediate dimension is also infinite. 
\end{remark}

\begin{remark}\label{remarksets}

Let us at this point briefly discuss this last proof. When proving these type of statements for the dimensions of sets, one usually starts by taking a cover of the set, and then either ''enlarge'' or ''subdivide'' the members of the cover appropriately in order to obtain a new cover, whose members have the diameters falling within a certain range. This is essentially what we did here, but there is an important aspect to note.

In the case of measures, it was not necessary to consider the ''whole set'', in the sense that we only had to study the behavior of balls centered at a fixed $x$. In other words, the dimensions of measures are in a sense ''local'', rather than ''global'' as in sets. This difference can be seen more easily with an example:

First, consider $F_p := \Set{\frac{1}{n^p}}\cup \Set{0}$, and the set:
$$C^n_p := \Set{x \in \Rn : |x| \in F_p} = \bigcup_k S^{n-1}_{1/k^p} = \bigcup_k S_k $$

\end{remark}

\begin{figure}[htbp]
\includegraphics[scale=.3]{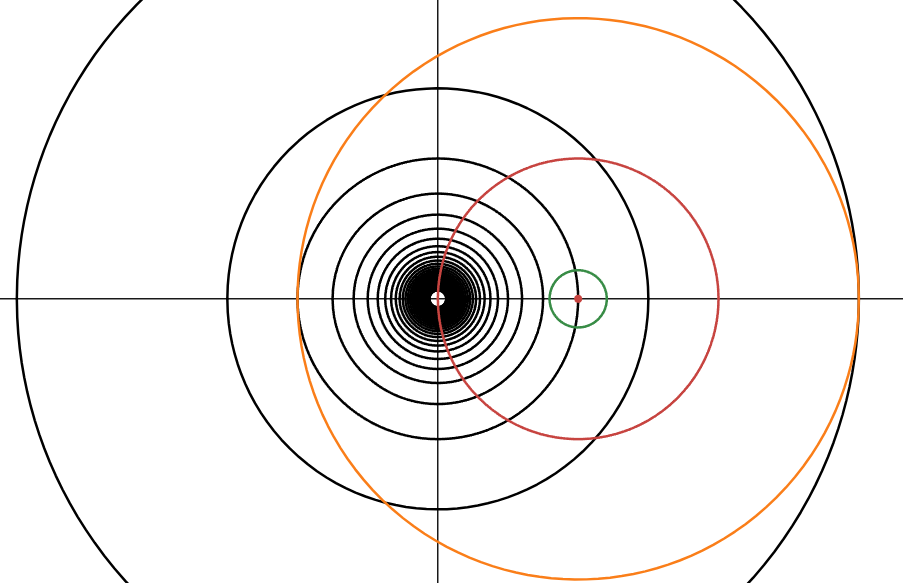}
%\end{center}
%\begin{center}
\caption{The set $C_{p}^2$ (in black) and some balls around one of its elements.}
\label{fig1}
\end{figure}

Now, suppose $p > \frac{1}{n-1}$, so that the measure $\mu = \sum \mathcal{H}^{n-1}_{|S_k}$ is finite. We will show that, in fact $\boxs(\mu) \leq n-1$: For this, take $x \in C_n^p$, so that $s \in S_k$ for some $k$ and $r>0$, and consider three cases: $r\leq 1/k^p$, $r\geq 2/k^p$ and $ 1/k^p \leq r \leq 2/k^p$ (see figure \ref{fig1}). For the first two cases, it is clear that $\mu(B(x, r)) \geq  c \cdot r^s$ for some constant $c>0$, and the same can be deduced for the third case by interpolating between the first two cases.

Moreover, just as easily one can obtain the same inequality for the upper Hausdorff dimension, $\dim^*_{H}(\mu) = n-1$. Then, all intermediate dimensions are equal to $n-1$. However, $\dim_H(C^n_p)= n-1$ and so we obtain:
$$\boxs(C^n_p) \leq \boxs(\mu) \leq n-1 = \dim_H(C^n_p).$$

Then, the box dimension of this set can be obtained via the intermediate dimensions of measures. This last part is especially interesting because even though the proof of this last fact can be done without measures, it would involve a not small amount of calculations (see \cite{tan2020intermediate}), whereas this proof is a simple geometric observation. 

\begin{corollary} \label{corcont}
The function $\theta \mapsto \udim (\mu)$ is continuous in $(0, 1]$.
\end{corollary}

Considering that, for sets, it is important to determine whether $\theta \mapsto \udim (\supp{\mu})$ is continuous at $0$ or not, it would be interesting to try to relate the continuity at $0$ for sets to the continuity at $0$ for measures. As $\dim_H^*(\mu) \leq \dim_H((\supp{\mu}))$, we first need to establish conditions that guarantee that we have $\dim_H^*(\mu) = \dim_H((\supp{\mu}))$ (otherwise, $\theta \mapsto \udim(\mu)$ is simply not continuous at $0$).
 
Using \cite[Proposition 10.10 ]{Falconertechniques}, which decomposes any measure into a sum $\mu = \sum_s \mu_s + \mu^D$, it is not difficult to verify that if one has $\mu^D = 0$, then $\dim_H^*(\mu) = \dim_H((\supp{\mu}))$. Consequently, it is clear that the continuity of $\theta \mapsto \udim (\mu)$ at $0$ implies the same for $\theta \mapsto \udim ((\supp{\mu}))$.

However, this is not an equivalence: even if $\dim_H^{*}(\mu) = \dim_H (\supp{\mu})$, we will see later an example of a measure $\mu$ such that $\te \rightarrow \udim (\mu)$ not continuous while $\te \rightarrow \udim(\supp{\mu})$ is continuous.

\vspace{0.3cm}

Before proving the main theorem, we wish to point out a couple of remarks.

\begin{remark}
\label{obsdesigualdadmedidas}
    Let $\mu, \nu$ be measures with equal support. If $\mu(B) \leq \nu(B)$ is true for every Borel set $B$, then $\udim(\nu) \leq \udim(\mu)$ for all $\theta \in [0, 1]$. Hence, clearly
   $$ \udim(\mu + \nu) \leq \min\Set{\udim(\mu), \udim(\nu)}.$$
\end{remark}

\begin{remark} \label{obsdeltan}

Taking a closer look at the condition in Definition \ref{defdimintmed}, It is not difficult to notice that it suffices to check the condition $\delta< \delta_0$ for a decreasing sequence $\delta_n \to 0$ such that $\delta_{n+1} \geq a \cdot \delta_n$ for some $a\in \mathbb{R}_{>0}$:
\end{remark}

\begin{proof}

Take $\delta \in [\delta_{n+1}, \delta_n]$. We need to verify that this $\delta$ satisfies the condition in \ref{defdimintmed}:

To show this, taking the $r \in[\delta_{n+1}^{1/\theta}, \delta_{n+1}]$ associated with $\delta_{n+1}$, there are two cases: If $r \in [\delta^{1/\theta}, \delta]$, then the same $r$ works for $\delta$. If not ($r < \delta^{1/\theta}$), then we have:

$$\mu(B(x, \delta^{1/\theta}))\geq \mu(B(x, r))\geq c\cdot r^s \geq c \cdot \delta_{n+1}^{s/\theta}\geq c\cdot a^{s/\theta} \cdot \delta_n^{s/\theta} \geq \widetilde{c} \cdot (\delta^{1/\theta})^s.$$
\end{proof}

\begin{theorem} \label{teoprincipal}
Let $E \subseteq \Rn$ be compact. Then, for every $\theta \in (0, 1]$,
$$\udim(E) = \min \Set{\udim (\mu) : \mu(\Rn)< \infty \text{ and } \supp{\mu}=E }.$$
Moreover, there exists a finite measure $\mu$ such that  $\supp{\mu}=E$, and
$$\udim (E) = \udim (\mu) \: , \: \forall \theta \in (0, 1].$$
\end{theorem}

\begin{proof}

We proceed to prove this in three steps:

\begin{enumerate}[label=\Roman*)]
    \item We first prove that for each $\theta$, we have $$\udim(E) = \inf \Set{\udim (\mu) : \mu(\Rn)< \infty \text{ and } \supp{\mu}=E}.$$

    \item We will show that this infimum is in fact a minimum.
    \item We will construct the measure $\mu$ that satifies the equality for all $\theta$.
\end{enumerate}

\begin{enumerate}[label=\Roman*)]

\item Fix $\theta$, by Proposition~\ref{musop1}, we have the inequality $\leq$. % and try to prove the theorem with an infimum first.
 Now, let $s > t > \udim(E)$. We will show that there exists $\mu$ with $\supp \mu = E$, such that $s>\udim(\mu)$.

Taking $\varepsilon=1$ in the definition of $\udim(E)$, it is clear that there exists $ \delta_0$ such that, for every $\delta<\delta_0$ there exists a cover $ \{U_i\}$ of $E$ with  $\delta^{1/\theta}\leq  \left| U_i \right|\leq \delta$, and $\sum  \left| U_i \right|^t<1$. Now, taking $n_0$ such that $\frac{1}{2^{n_0}}<\delta_0$, and $\delta_n = \frac{1}{2^n}$, we obtain that there
%\textcolor{red}{Now, taking $n_0$ such that $\frac{1}{2^{n_0}}<\delta_0$, and $\delta$ as $\frac{1}{2^n}$ (with $n\geq n_0$), the following can be deduced, taking $\delta_n = \frac{1}{2^n}$:}
exists $n_0$ such that for every $n \geq n_0$ we have a covering of $E$, $ \{ U_i^{(n)} \}_{i=1 , \dots , C_n}$, with:
$$\delta_n^{1/\theta}\leq  \left| U_i^{(n)} \right|\leq \delta_n \: , \: \text{ and } \: \sum_{i=1}^{C_n}  \left| U_i^{(n)} \right|^t < 1.$$

Now, there exists $  c>0$ such that for all $ n\geq n_0$
$$\sum | U_i^{(n)} |^s = \sum | U_i^{(n)} |^t | U_i^{(n)} |^{s-t} \leq \delta_n^{s-t} \sum | U_i^{(n)} |^t \leq \frac{1}{(2^{s-t})^n}\leq \frac{c}{n^2}. $$
Therefore, we can define the following measure $\mu$, taking for each $n, i$ some element $x_i^{(n)} \in U_i^{(n)} \bigcap E$:

$$\mu  := \sum_{n \geq n_0} \sum_{i=1}^{C_n} | U_i^{(n)} |^s \delta_{x_i^{(n)}}.$$

The measure $\mu$ is fully supported in $E$, since its support is contained in $E$ and for each $x \in E$ a sequence of points among the $x_i^{(n)}$ can be found such that it converges to $x$. It is a finite measure, because
$$\mu(\Rn) = \sum_{n \geq n_0} \sum_{i=1}^{C_n} | U_i^{(n)} |^s \leq \sum_{n \geq n_0} \frac{c}{n^2}< \infty.$$

Now, the only thing that needs to be verified for this step is that $s> \udim(\mu)$. For this, Observation~\ref{obsdeltan} will be helpful.

Taking the sequence $\delta_n = \frac{1}{2^n}$ ($n \geq n_0$), $n\geq n_0$ and $x \in \supp{\mu}=E$, there exists $i$ such that $x \in U_i^{(n)}$ and so we can take $r=| U_i^{(n)} | \in [\delta_n^{1/\theta}, \delta_n]$, so that 
$x_i^{(n)} \in B(x, r)$, and: 

$$\mu(B(x, r)) \geq \mu ({x_i^{(n)}}) \geq | U_i^{(n)} |^s = r^s.$$

Then, $s>\udim (\mu)$.

\item 
In order to show that the infimum is a minimum, we remark that multiplying a measure by a constant factor does not affect its intermediate dimensions. With this in mind, because of the previous step there is a sequence of measures $\mu_n$, all finite and fully supported in $E$, such that $\udim(\mu_n) \searrow \udim(E)$, and by the previous comment we can suppose that they also satisfy $\mu_n(\Rn) \leq \frac{1}{n^2}$. 

Thus, $\mu_\theta := \sum \mu_n$ is a finite measure, fully supported in $E$, and by Remark~\ref{obsdesigualdadmedidas} we conclude that:
 
$$\udim (\mu_\theta) = \udim (\sum \mu_n) \leq \udim (\mu_n) \: \forall \: n,$$

so that $\udim(E) \leq \udim (\mu_\theta) \leq \udim(E)$.

\item To construct the measure that verifies the equality in the statement of the theorem, for all $\theta$, first let $\{q_n\}_{\mathbb{N}} = \mathbb{Q} \bigcap (0, 1]$ be an ordering of the rationals in the unit interval. For each $n$ there exists $ \mu_{q_n}$ finite and fully supported in $E$ with $\overline{\dim}_{q_n}(\mu_{q_n}) = \overline{\dim}_{q_n}(E)$, and without loss of generality satisfying $\mu_{q_n} (\Rn) \leq 1/n^2$. Then, just like before. the measure:

$$\mu := \sum \mu_{q_n}$$ can be considered, which is finite, fully supported in $E$, and 
$$\overline{\dim}_{q_n}(E) \leq \overline{\dim}_{q_n} (\mu) \leq \overline{\dim}_{q_n} (\mu_{q_n}) = \overline{\dim}_{q_n}(E) \: \forall \: n.$$

In conclusion, $\udim(\mu) = \udim(E)$ for all $\theta \in \mathbb{Q}$. In addition, the functions $\theta \mapsto \udim(E)$ and $\theta \mapsto \udim(\mu)$ are both continuous in $(0, 1]$ and coincide in a dense subset of $(0, 1]$, so:
$$\udim(\mu) = \udim(E) \: \forall \: \theta \in (0, 1].$$
\end{enumerate}
\end{proof}

\begin{remark}
The proof given in \cite{Falconerminkowski} for $\boxs\mu$ did not rely on proving particular cases first: To be more specific, the definition of Minkowski dimensions for sets enabled the author to prove the minimum without having to demonstrate the infimum. However, because of our definition of intermediate dimensions, it was necessary to check the step I first.
\end{remark}

The importance of this theorem lies in the fact that it allows some results for dimensions of measures to be translated to results for dimensions of sets. For example, the inequality $\udimp (E) \leq \frac{\phi}{\theta}\udim(E)$ can be deduced directly from Proposition~\ref{propdesigualdad1} by taking the measure given by Theorem \ref{teoprincipal}.

\begin{proposition}
    Let $\mu$ be a measure such that $\dim_A(\mu)< \infty$ and $\theta \in(0, 1]$. Then,
    $$\udim(\mu) \geq \dim_A(\mu) - \frac{\dim_A(\mu) - \boxs(\mu)}{\theta}.$$
\end{proposition}

\begin{proof}

Just like in Remark \ref{remarksets}, this proof is analogous to the one given for sets in \cite{falconer2020intermediate}(Prop. 2.4). Take: $$0<b<\boxs(\mu) \leq \dim_A(\mu)<d<\infty$$ (if $\boxs(\mu) = 0$ then the result is trivial). Then, the following properties are satisfied:

\begin{enumerate}
    \item \label{primera} There exists $\widetilde{c}>0$ such that 
    $$\frac{\mu(B(x, r))}{\mu(B(x, R))}\geq \widetilde{c} \cdot
     \left ( \frac{r}{R} \right )^d \:, \: \text{for all } \: 0<r<R<1, x\in\supp{\mu}.$$

    \item \label{second} For all $ c, \delta_0 > 0 $ there exist $ x \in \supp{\mu},  \delta<\delta_0 \: $ such that $\: \mu(B(x, \delta))< c \cdot \delta^b$.  

\end{enumerate}

Therefore, we need to verify that $\udim(\mu) > d - \frac{d-b}{\theta}$, or in other words, that for all $c, \delta_0$ there exist $ \: \delta<\delta_0, x \in \supp{\mu}$ such that, for all $r\in[\delta^{1/\theta}, \delta]$, we have $\mu(B(x, r))< c \cdot r^{d - (d-b)/\theta}$:

Take $c, \delta_0>0$. Because of property (\ref{second}), there exists $ \:\delta<\delta_0$ with $$\mu(B(x, \delta))< c \cdot \delta^b.$$ Now, with these $\delta, x$, notice that for every $r \in[\delta^{1/\theta}, \delta]$ we can use property \eqref{primera} with $r$ and $\delta$, so that:

$$\widetilde{c} \cdot  \left ( \frac{\delta}{r} \right )^d \leq \frac{\mu(B(x, \delta))}{\mu(B(x, r))} < \left ( \frac{\delta}{r} \right )^d < \frac{c \cdot \delta^b}{\mu(B(x, r))},$$
and then, 
$$\mu(B(x, r))< \frac{c}{\widetilde{c}} \cdot \delta^{b-d}r^d \leq \frac{c}{\widetilde{c}} \cdot r^{d - (d-b)/\theta},$$
so the result follows.
\end{proof}

\begin{corollary} \label{cordegeq}
    If $\dim_A(\mu) = \boxs(\mu) < \infty$, then $\udim(\mu) = \dim_A(\mu) $ for all $ \theta \in (0, 1]$.
\end{corollary}

\begin{remark}
    This last corollary shows that it is possible that the function $\theta \mapsto \udim(\mu)$ is constant while $\theta \mapsto \udim(E)$ is not.

    For this, take $E = \Set{\frac{1}{n}}_{\mathbb{N}}$, and $\mu (\frac{1}{n})= \frac{1}{n^2}$. Then, by  \cite[theorem 4.1]{HARE_2020}, one concludes that:

\begin{itemize}
    \item $\dim_A(\mu) = \boxs(\mu) = 1$
    \item $\dim_{\theta}(E) = \frac{\theta}{1+\theta}$, for $\theta \in [0,1]$,
\end{itemize}
and hence $\theta \mapsto \udim(\mu)$ is constant.
\end{remark}

\begin{remark}

A question related to the second step in the proof of Theorem~\ref{teoprincipal} would be whether the function $\mu \mapsto \udim(\mu)$ (with fixed $\theta$) is continuous in the weak topology. However, this is not true, and it is not difficult to find a counterexample:

The idea behind the counterexample consists in taking the underlying set as $$F_1 =\Set{0} \bigcup \Set{\frac{1}{n}}_{\mathbb{N}},$$ and a measure supported in such set, and then 'disturb it' (this counterexample works with $\theta = 1$, in fact):

Following the calculations found in \cite[theorem 4.1]{HARE_2020}, taking the measures:

$$\mu := \sum \frac{1}{n^{3/2}} \delta_{1/n} \:, \: \mu_k := \mu\vert_{[1/(k-1), 1]} + \sum_{n=k}^{\infty} \frac{1}{n^{11/10}} \delta_{1/n}$$

we have that $\boxs(\mu)= \frac{3}{4}$, and $\boxs(\mu_k) = \frac{11}{20}$. However, $\mu_k  \rightharpoonup \mu$, hence the function $\mu \mapsto \boxs(\mu)$ is not continuous.

\end{remark}

We conclude this section with a direct calculation of the intermediate dimensions of certain discrete measures:

Following the notation in \cite{HARE_2020}, let $\lambda>0, \beta>1$, and for $k \in \mathbb{N}$ take $a_k := 1/k^\lambda$ and $p_k:=1/k^\beta$. We can take the set $E := \{  a_k \}_\mathbb{N} \cup \{0\}$ and the measure $\mu := \sum p_k \de_{a_k}$, and the following constants associated with these parameters:
$$s:=\frac{\beta -1}{\lambda} \: \: , \: \: t:= \frac{\beta}{\lambda + 1} \: \:,\:\: C(\te):= \frac{\beta\te}{\lambda} + (1-s) \:\:  \:\:, \: \: f(\te) :=\frac{\beta \te}{\beta \te + \lambda + 1 - \beta}= \frac{\beta\te}{\lambda C(\te)}.$$

Our goal is to prove the following result:

\begin{theorem} \label{lacuenta2}
Let $\te \in (0, 1]$. Then, for the measure $\mu$ defined above we have:
\begin{equation}\label{lacuenta}
\udim \mu= \left\{ \begin{array}{lcc} s & \text{if} & s\geq1 \\ \max \left\{ s, f(\te) \right\} & \text{if} &  s<1 \end{array} \right.\quad
\end{equation}

\end{theorem}

Before giving the proof, we state one lemma and some observations regarding the theorem. Throughout this last part of the section, $\mu$ will always be $\sum p_n \delta_{a_n}$.

%\begin{obs}
%The function$f(\te)$ is non-decreasing in $\te$ and $f(1)=t$, and a simple calculus shows that, when $s\geq1$ we have that $t\leq s$. Then, the theorem implies that in this case $$\udim\mu=s \text{ for all }\te>0.$$
%\end{obs}
\begin{obs}\label{obscuenta}
When $s<1$ it is easy to show that $t>s$, so in this case $f(1)>s$ and so $\udim \mu$ will be equal to $f(\te)$ in an interval containing $1$. Moreover, it is clear that, taking $\te= (\beta-1)\beta^{-1}$, we obtain $s=f(\te)$, implying that $$\max \left\{ s, f(\te) \right\} = f(\te)$$ in $[(\beta-1)\beta^{-1}, 1]$ and $\max \left\{ s, f(\te) \right\}=s$ outside that interval (\textcolor{red}{see figure \ref{fig2}}).

In fact, in the proof we will consider the cases $s\geq1$ and $s<1$ separately.
\end{obs}
\begin{figure}[htbp]
\includegraphics[scale=.55]{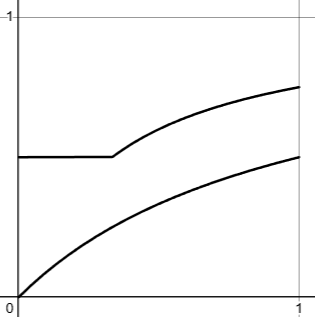}
\caption{$\udim\mu$ and $\udim E$ for $\lambda=1, \beta=3/2$.}
\label{fig2}
\end{figure}

Finally, we will need two lemmas, which are a direct consequence of lemmas 4.3 and 4.4 (\cite{HARE_2020}).

\begin{lemma} \label{lemadelacuenta}
For every $\varepsilon>0$ there exists $r_0$ such that, for all $r<r_0$ we have the following:
$$s-\varepsilon< \frac{\log\mu (B(a_k, r))}{\log r}<s+\varepsilon \:\:, \:\:\text{if  } r>a_k.$$
$$\frac{\log a_k}{\log r}(s-1) + 1-\varepsilon< \frac{\log\mu (B(a_k, r))}{\log r}<\frac{\log a_k}{\log r}(s-1) + 1+\varepsilon \:\:, \:\:\text{if   } a_k - a_{k+1}<r\leq a_k.$$
$$\frac{\log\mu(B(a_k, r))}{\log r} = \frac{\beta}{\lambda}\frac{\log a_k}{\log r} \text{ if } r \leq a_k - a_{k+1}.$$
\end{lemma}

With this in mind, we can proceed with the proof of Theorem \ref{lacuenta2}.

\begin{proof}

We will prove the result for $s\geq 1$ and $s < 1$ separately.

\underline{Case $s\geq 1:$} We need to prove that $\udim \mu = s$ for all $\theta>0$. Theorem 4.1 (\cite{HARE_2020}) implies that $\boxs \mu = s$, which means that it is enough to prove that, for all $\varepsilon>0$, $\udim \mu > s-\varepsilon$, or in other words, that:
\begin{equation*}
 \text{ For all } c, \de_0, \text{ there exist } \de<\de_0\:,\: a_k \text{ such that for } r\in [\de^{1/\theta}, \de]\:,\: \frac{\log \mu(B(a_k, r))}{\log r}> \frac{\log c}{\log r} + s-\varepsilon.   
\end{equation*}

But this is clear by the previous lemma, taking $a_k, \de$ sufficiently small and such that $a_k<\de^{1/\te}$.

\underline{Case $s< 1:$} We will first prove that $\udim \mu \geq \max\{ s, f(\te) \}$. 

Following the same reasoning as in the first case, it is clear that $\udim \mu \geq s$. However, Theorem 4.1(\cite{HARE_2020}) implies that $\boxs \mu =t$, so more care is needed. That said, we will see that for all $\varepsilon>0$, $\udim\mu \geq f(\te) - \varepsilon$.

Because of obs. \ref{obscuenta}, without loss of generality $\te > (\beta-1)\beta^{-1}$, which means that $C(\te)>1$. Now, take $c, \delta_0>0$. We need to prove that there exist $\de<\de_0$ and $a_k$ such that, for all $r\in[\de^{1/\te}, \de]$ we have $\mu (B(a_k, r)) < c \cdot r^{f(\te)-\varepsilon}$.

Take $a_k$ such that $\de:= a_k^{C(\te)}<\de_0$ and  $\de^{1/\te}\leq a_k-a_{k+1} =: b_k \leq \de$ (to check that such $a_k$ exists, note that $b_k \sim a_k^{(\lambda+1)/\lambda}$ and that $C(\te)< (\lambda+1)\lambda^{-1}< C(\te)/\te$), and such that we have the following two inequalities:
\begin{equation}
\frac{\log a_k}{\log r}(s-1) + 1-\varepsilon< \frac{\log\mu (B(a_k, r))}{\log r}<\frac{\log a_k}{\log r}(s-1) + 1+\varepsilon \:\:, \:\:\text{ for }  r \in(b_k, \de].
\end{equation}
\begin{equation}
\frac{\log\mu(B(a_k, r))}{\log r} = \frac{\beta}{\lambda}\frac{\log a_k}{\log r} \text{ for } r \in [\de^{1/\te}, b_k].
\end{equation}

(notice that, because $C(\te)>1$, $a_k^{C(\te)}=\de < a_k$, so the third case in lemma 
\ref{lemadelacuenta} does not appear). Then, we have to study two cases:
\begin{itemize}
    \item $\underline{r\in[b_k, \delta]}$: In this case, we have the following inequalities:
    \begin{align*}
    \frac{\log \mu(B(a_k, r))}{\log r} & \geq \frac{\log a_k}{\log r}(s-1)+1 - \varepsilon \geq \frac{\log a_k}{\log \de}(s-1)+1-\varepsilon = \frac{s-1}{C(\te)}+1-\varepsilon\\
    &=\frac{\beta \te / \lambda}{\frac{\beta\te}{\lambda}+1-s}-\varepsilon=\frac{\beta\te}{\beta\te + \lambda(1-s)}-\varepsilon = f(\te)-\varepsilon\\
    &\geq \frac{\log c}{\log r}+f(\te) - 2\varepsilon.
    \end{align*}
\item $\underline{r\in[\de^{1/\te}, b_k]}$: In this case,
\begin{align*}
\frac{\log \mu(B(a_k, r))}{\log r} &= \frac{\beta}{\lambda}\frac{\log a_k}{\log r}\geq \frac{\beta}{\lambda}\frac{\log a_k}{\log \de^{1/\te}}=\frac{\beta \te}{\lambda C(\te)} = f(\te) = f(\te) + \varepsilon - \varepsilon\\
& > \frac{\log c}{\log r} + f(\te) - \varepsilon.
\end{align*}
\end{itemize}

This means that $\udim \mu \geq f(\te)-\varepsilon$, and so $\udim\mu \geq \max\{ s, f(\te) \}$.

Now, we need to prove the other inequality, that
\begin{equation}\label{lacuentades}
\udim\mu < \max\{ s, f(\te) \}+\varepsilon \text{ for all } \varepsilon.
\end{equation}
In order to prove this, it is useful to remember that
$$\max\{ s, f(\te) \}=\left\{\begin{matrix}f(\theta), \text{ if }\te\geq (\beta-1)\beta^{-1}.
 \\s,\text{ if } \theta< (\beta-1)\beta^{-1}.
\end{matrix}\right.$$

Then, we will prove inequality \ref{lacuentades} for $\te>(\beta-1)\beta^{-1}$ and $\te\leq (\beta-1)\beta^{-1}$ separately.

\underline{for $\te>(\beta-1)\beta^{-1}$}:
We need to prove that $\udim \mu<f(\te)+\varepsilon$, so it is enough to show that: There exists $\de_0$ such that for all $\de<\de_0 \:,\: a_k$  there exists $r\in[\de^{1/\te}, \de] \:/\:  \mu(B(a_k, r)) \geq r^{f(\te)+\varepsilon}$, or in other words, that
\begin{equation*}
\frac{\log \mu(B(a_k, r))}{\log r}\leq f(\te) + \varepsilon.
\end{equation*}

For this, we will first check that this holds for $a_k$ sufficiently small, and then that, with that in mind, the condition is clear for all $a_k$.

Take $a_k, \de_0$ sufficiently small such that he following holds:
\begin{equation}\label{cuentaeq1}
\frac{\log a_k}{\log r}(s-1)+1-\varepsilon \leq \frac{\log \mu(B(a_k, r))}{\log r} \leq \frac{\log a_k}{\log r}(s-1)+1+\varepsilon \: , \: \forall \: r\: \in [b_k, a_k].
\end{equation}
\begin{equation}\label{cuentaeq2}
s-\varepsilon \leq \frac{\log \mu(B(a_k, r))}{\log r} \leq s + \varepsilon\:,\: \forall \: r \in[a_k, \delta_0]. 
\end{equation}

Then, to verify that the condition holds for these $a_k$, we study the following cases:
\begin{itemize}
    \item $a_k \leq \de:$
    Taking $r= \de$, we have that:
    \begin{equation*}
    \frac{\log \mu(B(a_k, r))}{\log r}\leq s+\varepsilon\leq f(\te) + \varepsilon.
    \end{equation*}
So the condition holds.
    \item $a_k^{C(\te)}\leq \de\leq a_k:$
    Taking $r=\de$, we obtain:
    \begin{equation*}
    \frac{\log \mu(B(a_k, r))}{\log r}\leq \frac{\log a_k}{\log r}(s-1)+1+\varepsilon \leq \frac{\log a_k}{C(\te) \log a_k}(s-1)+1+\varepsilon=f(\te) + \varepsilon.
    \end{equation*}
    \item $\de \leq a_k^{C(\te)}:$ In this case, $\de^{1/\te}\leq a_k^{C(\te)/\te}\leq b_k$, so we can take $r=\de^{1/\te)}$ to obtain:
    \begin{equation*}
    \frac{\log \mu(B(a_k, r))}{\log r}=\frac{\beta \te}{\lambda}\frac{\log a_k}{\log \de}\leq \frac{\beta \te}{\lambda C(\te)}=f(\te) \leq f(\te) + \varepsilon.
    \end{equation*}
\end{itemize}

This means that for sufficiently small $a_k$, te condition holds.

Now, let $a_{k_0}$ be such that the condition holds. First, notice that we can take $\de_0$ such that $\de_0<b_{k_0}$. Also, $a_k-a_{k+1}\leq a_{k-1}-a_k$, which means that for $k<k_0$, $B(a_k, r)$ will contain only a single element (namely, $a_k$), and so we have that, for these $k<k_0$ and $r<\de_0$,
$$\mu(B(a_k, r))=c_k\geq c_k r^{f(\te)
+\varepsilon}.$$
Therefore, taking $c:= \min\{1,c_1, \dots c_{k_0-1}\}$ we conclude that $\udim\mu \leq f(\te)+ \varepsilon.$

\underline{for $\te \leq(\beta-1)\beta^{-1}$}: We need to check that $\udim \mu \leq s+\varepsilon$. Just like in the other case, it is enough to show that: There exists $\de_0$ such that for all $\de<\de_0 \:,\: a_k$  there exists $r\in[\de^{1/\te}, \de] \:/\:  \mu(B(a_k, r)) \geq r^{s+\varepsilon}$, or in other words, that
\begin{equation*}
\frac{\log \mu(B(a_k, r))}{\log r}\leq s + \varepsilon.
\end{equation*}

Just like before, take $a_k, \de_0$ such that \eqref{cuentaeq1} and \eqref{cuentaeq2} hold.

Then, we have to study the following cases:
\begin{itemize}
    \item $a_k\leq \de$: This case follows just like before.
    \item $a_k > \de$: First, note that $C(\te)<1$, so $\de<a_k^{C(\te)}$, and so (for $a_k$ sufficiently small), $\de^{1/\te}<a_k^{C(\te)/\te}\leq b_k$ (because $b_k \sim a_k^{\lambda+1 / \lambda}$). Then, taking $r=\de^{1/\te}$,
    \begin{equation*}
    \frac{\log \mu(B(a_k, r))}{\log r}= \frac{\beta}{\lambda}\frac{\log a_k}{\log \de^{1/\te}}=\frac{\beta \te}{\lambda}\frac{\log a_k}{\log \de}\leq \frac{\beta \te}{\lambda}\leq s.
    \end{equation*}
\end{itemize}

By the same argument as before, we conclude that for all $\te < (\beta-1)\beta^{-1}$, we have $\udim\mu\leq s+ \varepsilon$ for all $\varepsilon>0$.

\end{proof}

\section{Products and dimensions}

    Given two measures $\mu, \nu$, we have that $\supp{\mu \times \nu} = \supp{\mu} \times \supp{\nu}$. It follows that if both are finite and compactly supported, then the product measure possesses those properties as well. Then we can consider products of measures and their dimensions.

\begin{theorem}
    The following inequalities are true:
    $$\ldim (\mu) + \ldim (\nu) \leq \ldim(\mu \times \nu) \leq \udim(\mu \times \nu)\leq \udim (\mu) + \boxs(\nu).$$
\end{theorem}

\begin{proof}

We will prove the last inequality first:

Let $s> \udim(\mu), t>\boxs(\nu)$. Then:

\begin{enumerate}
   \item There exists $c$ such that $ \nu(B(y, r))\geq c \cdot r^t $ for all $ y\in \supp{\nu}, r \in (0,1).$
    \item There exist $ \widetilde{c}$ and $\delta_0 >0$ such that for all $ \delta<\delta_0, x \in \supp{\mu} , $ there exists $\: \widetilde{r} \in [\delta^{1/\theta}, \delta]$ with $\mu(B(x, \widetilde{r}))\geq \widetilde{c} \cdot \widetilde{r}^s.$
\end{enumerate}

In order to check that $s+t \geq \udim(\mu \times \nu)$, first take $c \cdot \widetilde{c} \; , \text{ and } \; 2\delta_0$.

If $\delta<2\delta_0$, and $(x, y) \in \supp{\mu \times \nu} = \supp{\mu} \times \supp{\nu}$, it is clear that there exists $ \widetilde{r} \in [(\delta/2)^{1/\theta}, \delta/2]$ such that:
\begin{align*}
    (\mu \times \nu) \left ( B((x, y), 2 \cdot \widetilde{r}) \right ) & \geq (\mu \times \nu)\left (  B(x, \widetilde{r}) \times B(y, \widetilde{r})  \right ) = \mu(B(x, \widetilde{r})) \cdot \nu (B(y, \widetilde{r})) \\
    & \geq c \cdot \widetilde{c} \cdot \widetilde{r}^s \cdot \widetilde{r}^t = c \cdot \widetilde{c} \cdot (\widetilde{r})^{s+t} = \frac{c \cdot \widetilde{c}}{2^{s+t}} (2 \cdot \widetilde{r})^{s+t}.
\end{align*}

However, given that $\frac{2}{2^{1/\theta}} \delta^{1/\theta} \leq 2 \cdot \widetilde{r}\leq \delta$, and $\frac{2}{2^{1/\theta}} \leq 1$, it is possible that $2\widetilde{r}$ does not fall within the ''appropriate range''. However, just like before, this can be fixed quite easily:

If $2 \cdot \widetilde{r} \in [\delta^{1/\theta}, \delta]$, we are done. Otherwise, taking $\delta^{1/\theta}$, we have the following:
\begin{align*}
    (\mu \times \nu)(B((x, y), \delta^{1/\theta}))\geq (\mu \times \nu)(B((x, y), 2 \cdot \widetilde{r}))&\geq \frac{c \cdot \widetilde{c}}{2^{s+t}} (2 \cdot \widetilde{r})^{s+t}\\
    &\geq \frac{c \cdot \widetilde{c}}{2^{s+t}} \left ( \frac{2}{2^{1/\theta}} \right )^{s+t} \left ( \delta^{1/\theta} \right )^{s+t}.
\end{align*}

In conclusion, for each pair $\delta<2\delta_0$ and $(x, y)$ there exists $r \in [\delta^{1/\theta}, \delta]$ such that:
$$(\mu \times \nu) (B((x, y), r)) \geq \frac{c \cdot \widetilde{c}}{2^{s+t}} \left ( \frac{2}{2^{1/\theta}} \right )^{s+t} r^{s+t},$$

and so $s+t \geq \udim(\mu \times \nu)$.

For the inequality $\ldim (\mu) + \ldim (\nu) \leq \ldim(\mu \times \nu)$, it is useful to notice that the following two conditions are equivalent:
\begin{enumerate}
    \item For all $c>0 \: , \: \delta_n \rightarrow 0$ there exist $\: n \in \mathbb{N} \: , \: \text{and } x \in \supp{\mu}$ such that for every $r \in [\delta_n^{1/\theta}, \delta_n], \: \mu(B(x, r)) < c \cdot r^s$.
    \item For all $c>0$ there exists $\delta_0$ such that for all $\delta<\delta_0$ there exists $x \in\supp\mu$ such that for every $r\in[\delta^{1/\theta}, \delta], \:  \mu(B(x, r)) < c \cdot r^s.$
\end{enumerate}

%\begin{equation*}
%  \forall \: c>0 \: \exists \: \delta_0 / \: \forall \: \delta < \delta_0 \: \exists \: x \in \supp{\mu}  \text{ such that } \: \forall \: r\in[\delta^{1/\theta}, \delta], \:  \mu(B(x, r)) < c \cdot r^s.  
%\end{equation*}

Then, take $s < \ldim (\mu)$, and $t < \ldim(\nu)$ as usual. We will prove the inequality $s+t<\ldim(\mu \times \nu)$. By definition, we have the following:

 For all $ c>0 $ there exists $ 
\delta_0 >0 $ such that for all $  \: \delta < \delta_0 $ there exists $ \: x \in \supp{\mu} $ and $ \: y \in \supp{\nu} $ such that, for every $ \: r\in[\delta^{1/\theta}, \delta],$
$$\mu(B(x, r)) < c \cdot r^s \; , \: \nu(B(y, r)) < c \cdot r^t.$$
%    \item For all $ c>0 $ there exists $ 
%     \wig{\delta_0} >0 $ such that for all $  \: \delta < \delta_0 $ there exists $ \: y \in \supp{\nu} $ such that for every $ \: r \in[\delta^{1/\theta}, \delta], \:  \nu(B(y, r)) < c \cdot r^t.$

We need to prove that the same is true for the measure $\mu\times\nu$, with $s+t$.

For this, let $\wig{c}>0$ and take $c:= \sqrt{\wig{c}}$. Then, we can work with the $\delta_0>0$ corresponding to the last property. Then, the following is clear:

For every $\delta < \delta_0$ there exist $x \in \supp{\mu}, y \in \supp{\nu}$ such that, for $r\in[\delta^{1/\theta}, \delta],$
$$\mu(B(x, r))< c \cdot r^s \text{ and }\nu(B(y, r))< c \cdot r^t.$$

%{\color{red}
%$\forall \: \delta < \wig{\wig{\delta_0}} \: \: \exists \: \: x \in \supp{\mu}, y \in \supp{\nu} \big/ \: \forall \: \: r\in[\delta^{1/\theta}, \delta] ,
%    \begin{array}{c}
%     \mu(B(x, r))< c \cdot r^s \\*[1mm]
%     \nu(B(x, r))< \wig{c} \cdot r^t
%    \end{array}$
%}
%\vspace{1.5mm}
Then, taking $(x, y)$ in the product space, we have that for all $ r\in[\delta^{1/\theta}, \delta],$
$$(\mu \times \nu) (B((x, y), r))\leq (\mu \times \nu) ( B(x, r)\times B(y, r) ) = \mu(B(x, r)) \cdot \nu (B(y, r)) < \wig{c} \cdot r^{s+t}.$$
\end{proof}

\begin{remark}

The proof of the first inequality for sets given in \cite{falconer2020intermediate} was fundamentally different, in the sense that it was proven using the Frostman lemma for intermediate dimensions of sets (Prop. 2.3).

\end{remark}

As a final comment, we would like to add a little lemma regarding generalised intermediate dimensions for measures, defined in the same fashion as for intermediate dimensions for sets, using an admissible function $\Phi (\delta)$  as defined in Definition~\ref{defdimintgener}.%and replacing $\delta^{1/\theta}$ with $\Phi (\delta)$ in the definitions of $\udim(\mu)$ and $\ldim(\mu)$. 

\begin{definition} \label{defgeneral}
 Let $\Phi$ be an admissible function. The \textbf{$\Phi$-intermediate dimensions} of $\mu$ are defined by:
    \begin{itemize}
        \item $\dgu (\mu) := \inf \Set{s \geq 0:
    \begin{array}{c}
     \: \exists \: c>0 \text{ and } \: \delta_0 > 0 / \forall \: \delta < \delta_0 \text{ and } x \in \supp{\mu},\:\\*[1mm]
     \exists \: r \in [\Phi(\delta), \delta] \text{ with } \: \mu (B(x, r)) \geq c \cdot r^s
     \end{array} }.$

        \item $\dgl (\mu) := \inf \Set{s \geq 0:
    \begin{array}{c}
     \: \exists \: c>0 \text{ and } \: \delta_n \to 0 / \forall \: n \in \mathbb{N} \text{ and } x \in \supp{\mu},\:\\*[1mm]
     \exists \: r \in [\Phi(\delta), \delta_n] \text{ with } \: \mu (B(x, r)) \geq c \cdot r^s
     \end{array} }.$
    \end{itemize}
\end{definition}

Most of what was said about intermediate dimensions holds for these definitions without much variation in the proofs, provided that there exists $0<c<1$ such that $c\cdot \Phi(2\de)\leq \Phi(\de)$, or  a similar condition holds (which is not necessary for the definition but is important for certain arguments regarding adjusting ranges, which appear in many proofs). In fact, we could have defined these generalised dimensions by taking the condition in definition \ref{defgeneral} as being true for $\mu$-almost all $x\in \supp{\mu}$. 

For the $\Phi$-intermediate dimensions that have this ''doubling'' condition in $\Phi$, Theorem \ref{teoprincipal} is true. That is, for every $\phi$ there exists a measure $\mu$ such that the dimensions of the set and the measure are equal.

On another note, given that $\Phi(\delta)= \delta$ is not admissible, it is not immediately clear that the Minkowski dimensions of a measure can be understood as generalised intermediate dimensions. However, just like in the case for sets (\cite[Prop. 3.4]{Banaji}), it is true.

\begin{proposition}
Let  $\Phi$ be an admissible function such that $\lim_{\delta \to 0} \frac{\log \delta}{\log \Phi (\delta)}=1$, and $\mu$ a measure. Then, $\boxs (\mu) = \dgu (\mu)$.
\end{proposition}

\begin{proof}

The inequality $\boxs\mu\geq \dgu\mu$ is clear. For the other one, we will show that for all $\theta<1$, $\udim (\mu) \leq \dgu (\mu)$, so that the result follows from Corollary \ref{corcont}.

If the dimension is $\infty$, then the inequality is clear. Else, let $\theta<1$ and $s> \dgu(\mu)$ ( so that there exist $c$ and $\delta_0$ such that for all
$ \delta<\delta_0, \: x\in \supp{\mu} \: $ there exists $ r \in [\Phi(\delta), \delta] \: $ with $ \: \mu(B(x, r)) \geq c r^s.$

By our hypothesis, it is clear that for sufficiently small $\delta$, we have
$$\frac{\log \delta}{\log \Phi (\delta)} \geq \theta\text{, so that }\log(\delta) \leq \log(\Phi(\delta)^\theta) \text{ and } \delta \leq \Phi(\delta)^\theta.$$ But then, 
$$\delta^{1/\theta}\leq \Phi(\delta)\leq \delta,$$ so we can conclude that 
$$\udim (\mu) \leq s\text{ and then }\udim(\mu) \leq \dgu(\mu).$$

\end{proof}

\begin{obs}
A classic example of such $\Phi(\de)$ is $\Phi(\de) := \frac{\de}{-\log \de}$.
\end{obs}

To conclude this section, we would like to study whether these generalised dimensions 'close the gap' between $\dim_H^{*}\mu$ and $\boxs\mu$. Theorem 6.1 \cite{Banaji} proves that the generalised dimensions for sets effectively interpolate between $\dim_H(E)$ and $\boxs(E)$ in the sense that for every $a\in[\dim_H(E), \boxs(E)]$ there exists an admissible function $\Phi_a$ such that $\overline{\dim}^{\Phi_a}(E)=a$, so it is natural to ask if the same holds for the generalised dimensions of measures. However, this is not true, as the next example shows:

\begin{example}
Let $\mu$ be as in Theorem \ref{lacuenta2}, that is, $\mu = \sum p_k \de_{a_k}$ with $a_k := 1/k^\lambda$ and $p_k:=1/k^\beta$, and take $s= (\beta - 1)/\lambda$ like before. Then, $\dim_H^* \mu =0$ because $\dim_H (\{a_k\})=0$, and $\overline{\dim_B}\mu \geq s$. However, no number in the gap $(0, s)$ is attained by the generalised dimensions. For every admissible function $\Phi$ we have
$$\dgu\mu\geq s.$$
In fact, the proof is identical to that of Theorem \ref{lacuenta2}, since it relies on choosing $a_k$ and $\de$ sufficiently small so that $a_k < \de^{1/\te}$. The same argument applies here by instead choosing $a_k < \Phi(\de)$.
\end{example}

%\begin{example}
%Let $\mu$ be as in Theorem \ref{lacuenta2}. Then, $\dgu\mu\geq s$ for all admissible functions $\Phi$. In fact, the proof is identical to that of Theorem \ref{lacuenta2}, since it relies on choosing $a_k$ and $\de$ sufficiently small so that $a_k < \de^{1/\te}$. The same argument applies here by instead choosing $a_k < \Phi(\de)$.
%\end{example}

\section{A Capacity-theoretic approach}\label{AAAB}

In this section we will study an alternative definition of intermediate dimensions in terms of capacities with respect to certain kernels, which will be useful for obtaining properties regarding the intermediate dimensions of pushforward measures by projections (Theorem \ref{AAA}), which we proceed to define: 

\begin{definition}

Let $f: \Rn \rightarrow \mathbb{R}^m$ continuous be a function, $\mu$ a Radon measure in $X$. The \textbf{pushforward measure} $\mu_f$ is defined by:
$$\mu_f(E) := \mu(f^{-1}(E)).$$
\end{definition}

In particular, we will study the measures $\mu_V:= \mu_{P_V},$ with $P_V$ the orthogonal projection from $\Rn$ onto some $V \in G(n, m)$, and the relations between the intermediate dimensions of $\mu_V$ and $\mu.$ Our analysis will be heavily inspired by the proofs given in \cite{projections}.

We first need to define two functions that play an important role in the projections of sets, which have been studied in \cite{projections}.

\begin{definition}\label{defcapacity}
Let $\theta \in (0, 1]$, $m \in \{1 , \dots, n \}$, $ s \in [0, m]$ and $r \in (0, 1)$. Then, we define the functions $\phi_{r,\theta}^{s,m}, \widetilde{\phi}_{r,\theta}^{s} : \Rn \rightarrow \mathbb{R}$ by:
  $$ \phi_{r,\theta}^{s,m}(x)= \left\{ \begin{array}{lcc} 1 & \text{if} & |x| \leq r \\ \\ r^{s}/|x|^{s} & \text{if} & r < |x| \leq r^{\theta} \\ \\ r^{\theta(m-s)+s}/|x|^{m} & \text{if} &  r^{\theta}\leq |x| \end{array} \right.\quad 
   \widetilde{\phi}_{r,\theta}^{s}(x)= \left\{ \begin{array}{lcc} 1 & \text{if} & |x| \leq r \\ \\ \frac{r^{s}}{|x|^{s}} & \text{if} & r < |x| \leq r^{\theta} \\ \\ 0 & \text{if} &  r^{\theta}\leq |x| \end{array}, \right.$$
   and the ``capacity'' of a measure as
   $$(C_{r,\theta}^{s,m}(\mu))^{-1}=\displaystyle\inf_{x\in \supp \mu}\int \phi_{r,\theta}^{s,m}(x-y)d\mu(y).$$
\end{definition}

%Throughout this section, by  $ \phi_{r,\theta}^{s,m}(x)$ and $ \widetilde{\phi}_{r,\theta}^{s}(x)$ we mean the functions (3.1) and (5.4) defined in \cite{projections}.

%\begin{definition}

%   $$\phi_{r,\theta}^{s,m}(x)= \min \left\{ 1, \frac{r^{s}}{|x|^{s}}, \frac{r^{\theta(m-s)+s}}{|x|^{m}}\right\} $$
%\end{definition}

%We also need to define a notion similar to that of ''capacity'' for sets:

%\begin{definition}\label{defcapacity}
 %   $(C_{r,\theta}^{s,m}(\mu))^{-1}=\displaystyle\inf_{x\in \supp \mu}\int \phi_{r,\theta}^{s,m}(x-y)d\mu(y).$
%\end{definition}

The proof of the next lemma is fairly straightforward:

\begin{lemma}\label{nico}
Let $\mu$ be a finite Borel measure and let $\theta\in(0,1] $. If $0<r<1$, then for all $0\leq t\leq s\leq m,$ 
\begin{equation*}
    -(s-t)\leq \left(\frac{\log C_{r,\theta}^{s,m}(\mu)}{-\log r}-s\right)-\left(\frac{\log C_{r,\theta}^{t,m}(\mu)}{-\log r}-t\right)\leq -\theta (s-t).
\end{equation*}
\end{lemma} 
\begin{proof}
Notice that
\begin{equation*}
\phi_{r,\theta}^{s,m}(x)\leq \phi_{r,\theta}^{t,m}(x)\leq r^{(t-s)(1-\theta)}\phi_{r,\theta}^{s,m}(x),
\end{equation*}
and so we have:
\begin{equation*}
C_{r,\theta}^{s,m}(\mu)\geq C_{r,\theta}^{t,m}(\mu) \geq r^{(s-t)(1-\theta)}C_{r,\theta}^{s,m}(\mu),
\end{equation*}
and the result follows
\end{proof}

\begin{remark}\label{remarkexist}
From this lemma it is clear that the function $s \mapsto \limsup_{r\to 0}\frac{\log C_{r,\theta}^{s,m}(\mu)}{-\log r}-s$ is strictly decreasing and continuous.    
\end{remark}

For this section we will also need the following lemma, that shows that the intermediate dimensions of measures and the kernels $ \widetilde{\phi}_{r,\theta}^{s}(x)$ are closely related.

\begin{lemma}\label{primerlema}
Let $\mu$ be a Borel measure and suppose that there exists $c>0$ such that for all $r_0>0$  there exists $r\in(0,r_0)$ such that 
  \begin{equation} \label{Desigualdad1}
\inf_{x\in \supp \mu} \int \widetilde{\phi}_{r,\theta}^{s}(x-y) d\mu(y) \leq c r^{s}.
  \end{equation}
Then $\udim \mu\geq s$. 

Furthermore, if the inequality \eqref{Desigualdad1} is satisfied  for all $r\in(0,1)$ then  $\ldim{\mu} \geq s$.

\end{lemma}

\begin{proof}
We will only prove the $\udim$ part and leave the $\ldim$ to the reader. First, it is clear that if we take $x\in \supp{\mu}$ and $\delta\in[r,r^\theta]$, we have:

$$\mu(B(x,\delta))\left(\frac{r}{\delta}\right)^s \leq \int\left(\frac{r}{\delta}\right)^s  \mathbbm{1}_{B(0,\delta)}(x-y)d\mu(y)\leq  \int \widetilde{\phi}_{r,\theta}^{s}(x-y) d\mu(y).$$ 

Then, %using inequality \eqref{Desigualdad1} 
by hypothesis it is clear that there exists a sequence $r_k \to 0$  and points $x_k\in \supp \mu$ such that  $\mu(B(x_k,\delta))\leq c\delta^s$ for all $\delta\in[r_k,r_k^\theta].$ 

Now, fix $\varepsilon>0$. For all $\widetilde{c}>0$, we have that $\delta^s \leq \widetilde{c} \cdot \delta^{s-\varepsilon}$ for $\delta$ sufficiently small, which implies that $s - \varepsilon<\udim \mu$, and so $\udim \mu \geq s$.
\end{proof}

\begin{remark}\label{tecnicismolower}
For the lower intermediate dimensions, we required in the previous lemma that \eqref{Desigualdad1} is true for all $r$ in an interval. In truth, it is enough that \eqref{Desigualdad1} is true for a geometric sequence, such as $r_n = 2^{-n}$. The way the lemma is stated better shows the contrast between the upper and lower dimensions, though.
\end{remark}

A useful tool in the analysis of the intermediate dimensions of projections of sets is centered around expressing the dimensions in terms of certain ''capacities''. The next Theorem is the measure version of this:

\begin{theorem}\label{teocap}
 Let $\theta\in(0,1]$, $\mu$ a Borel measure with $\udim\mu >0$, and $m>\dim_A \mu$. Then, there exist a unique $s_0 \in [0, m]$ and a unique $t_0\in [0, m]$ such that:
\begin{equation*}
\limsup_{r\to 0}\frac{\log C_{r,\theta}^{s_0,m}(\mu)}{-\log r}-s_0=0 \; \; , \: \: \liminf_{r\to 0}\frac{\log C_{r,\theta}^{t_0,m}(\mu)}{-\log r}-t_0=0.    
\end{equation*}

Furthermore, 
\begin{equation}\label{igualdadconcapacidades}
\udim\mu = s_0 \:\: \text{and} \:\: \ldim\mu=t_0.
\end{equation}

%    $$\udim \mu =\left\{\text{ the unique $s\in [0,m]$ such that $\limsup_{r\to 0}\frac{\log C_{r,\theta}^{s,m}(\mu)}{-\log r}-s=0$} \right\}$$ and 
%    $$\ldim \mu =\left\{\text{ the unique $s\in [0,m]$ such that $\liminf_{r\to 0}\frac{\log C_{r,\theta}^{s,m}(\mu)}{-\log r}-s=0$} \right\}.$$

\end{theorem}

\begin{proof}

Like before, we only prove the first equality. The uniqueness of such $s_0$ can be easily deduced from Remark \ref{remarkexist}. The existence (and the equality \eqref{igualdadconcapacidades}) can be easily deduced from the following two arguments:

First, suppose that $\limsup_{r\to 0}\frac{\log C_{r,\theta}^{s,m}(\mu)}{-\log r}-s>0$. Then, for all $c>0$  there exists a subsequence $\{r_k\}$  that tends to $0$ and a sequence of points in the support of $\mu$ , $\{x_k\}$, such that:
$$\int \phi_{r_k,\theta}^{s,m}(x_k-y)d\mu(y)\leq c r_k^s.$$

Then for all $\delta\in[r_k,r_k^{\theta}]$ we have 
$$\mu(B(x_k,\delta))=\int\mathbbm{1}_{B(0,\delta)}(x_k -y)d\mu(y)\leq \left(\frac{\delta}{r_k}\right)^s \int\phi_{r_k,\theta}^{s,m}(x_k-y)d\mu(y)\leq C\delta^{s},$$ which implies that $s\leq \udim\mu$.

Conversely, suppose that $\limsup_{r\to 0}\frac{\log C_{r,\theta}^{s,m}(\mu)}{-\log r}-s<0$. By definition, there exists $\de>0$ such that, for every $r\leq \de$, $$r^s<\inf_{x\in\supp{\mu}}\int \phi_{r,\theta}^{s,m} (x-y)d\mu(y).$$

Then, observing that for every pair $r\leq \wig{r}$ we have $\phi_{r,\theta}^{s,m}(x)\leq \phi_{\wig{r},\theta}^{s,m}(x)$ we obtain that for some $C>0$ independent of $x$ and $r$,
    \begin{equation}\label{ecuacion1}
        Cr^s\leq \int \phi_{r,\theta}^{s,m}(x-y)d\mu(y) \:\:,\:\: \forall \: r \in (0, 1) \; , \; x\in\supp{\mu}.
    \end{equation} 
%for every $r\in (0,1)$ and every $x\in \supp\mu.$ 

Now, without loss of generality we can suppose $|\supp \mu|=1.$ Let $D=\lceil\log_2 (1/r)\rceil$ and $M$ the unique integer satisfying $2^{M-1}r<r^\theta\leq 2^M r$, and also let $r_0$ sufficiently small to ensure that $2\leq M\leq D-2$ for all $0<r\leq r_0$. 

Taking $x\in \supp\mu$, using inequality  \eqref{ecuacion1} and partitioning the support of $\mu$ on consecutive annuli  of the form $B(x,2^k r)\setminus B(x,2^{k-1}r)$, $(1\leq k\leq D)$, we obtain the following:

\begin{align*}
        Cr^s&\leq \int \phi_{r,\theta}^{s,m}(x-y)d\mu(y)\\
        &\leq \mu(B(x,r))+\sum_{k=1}^{D}\int_{ B(x,2^k r)\setminus B(x,2^{k-1}r) }\phi_{r,\theta}^{s,m}(x-y)d\mu(y) \\
        &\leq \mu(B(x,r))+\sum_{k=1}^{M}\int_{ B(x,2^k r)\setminus B(x,2^{k-1}r) }2^{-(k-1)s}d\mu(y)\\
        &\quad + \sum_{k=M+1}^{D}\int_{ B(x,2^k r)\setminus B(x,2^{k-1}r) }r^{\theta(m-s)+s}(2^{k-1}r)^{-m}d\mu(y)\\
        &\leq \mu(B(x,r))+\sum_{k=1}^{M}\mu(B(x,2^k r))2^{-(k-1)s}d\mu(y)\\
        & \quad +\sum_{k=M+1}^{D} \mu(B(x,2^k r))r^{(\theta-1)(m-s)}2^{-m(k-1)}.\\
    \end{align*}

 Then for each $x$ in the support of $\mu$, at least one of the terms of the above sums  is at least the arithmetic mean of the total sum. Therefore, we need to consider three cases:

\begin{enumerate}
    \item $\displaystyle \frac{Cr^s}{D+1}\leq \mu(B(x,r)). $
    \item $\displaystyle\frac{Cr^s}{D+1} \leq 2^{s} \mu(B(x,2^{k}r))2^{-ks}=4^s \mu(B(x,2^{k}r)) |B(x,2^{k}r)|^{-s}r^s $
for some $k\in\{1,...,M\}.$
    \item One term of the last sum is bigger or equal than the mean, i.e. (noticing that, by hypothesis, $m>\dim_A \mu$) for some $k\in \{M+1,...,D\}$ and a constant $c>0$,
    \begin{align*}
\frac{Cr^s}{D+1}&\leq r^{(\theta-1)(m-s)}\mu(B(x,2^{k}r))2^{-(k-1)m}\\
&\leq   r^{(\theta-1)(m-s)} \left(\frac{c2^{k}r}{r^\theta}\right)^{m}\mu(B(x,r^\theta))2^{-(k-1)m}\\
&\leq \left(2c\right)^m r^{-\theta s} r^{s}\mu(B(x,r^\theta).
\end{align*}
\end{enumerate}

Then, since for all $\varepsilon>0$ we have $\displaystyle 1\leq \frac{r^{-\varepsilon}}{D+1}$ for all sufficiently small $r$, we have that for all $\varepsilon>0$ there exist $C>0$ and $r_0$ such that for all $r\in(0,r_0)$ and all $x\in \supp\mu$ there exists $\delta\in[r,r^{\theta}]$ such that $C\delta^{s+\varepsilon}\leq \mu(x,\delta)$, which implies
$$\udim \mu \leq s+\varepsilon$$
and so the result follows by letting $\varepsilon$ tend to $0$.
\end{proof}

We now proceed to establish relations between the intermediate dimensions of $\mu$ and the push-forward measures $\mu_V$. First of all, a useful lemma will be stated:

\begin{lemma}(\cite[Lemma 5.3]{projections})\label{lemadelkernel}
For all $m\in\{1,...,n-1\}$ and $0\leq s<m$ there exist positive constants $a:=a_{n,m}$ and $b:=b_{n,m}$, depending only on $s$ such that for all $x\in\Rn$, $\theta\in(0,1]$ and $0<r<1/2$,
$$a_{n,m}\int \widetilde{\phi}_{r,\theta}^s (P_V x)d\gamma_{n,m}(V)\leq \phi_{r,\theta}^{s,m}(x)\leq b_{n,m}\int \widetilde{\phi}_{r,\theta}^s (P_V x)d\gamma_{n,m}(V).$$
\end{lemma}

Given that projections are in fact Lipschitz functions, the next lemma can be directly applied to projections:

\begin{lemma}
    Let $\mu$ be a Borel measure in $\Rn$ (we always assume that $\mu$ is finite, with compact support), and $f: \Rn \rightarrow \mathbb{R}^m$ a Lipschitz function. Then, $$\udim(\mu_f) \leq \udim (\mu)$$ with $\mu_f$ the push-forward measure. The same result is true for the lower intermediate dimensions.
\end{lemma}

\begin{proof}
First notice that (by Theorem 1.18, \cite{mattilafrac}), we have that $$\supp{\mu_f}= f(\supp \mu),$$ and let $\lambda>1$ such that $d(f(x), f(y))\leq \lambda d(x, y).$ We also have:
$$f^{-1}(B(f(x), \lambda  r))  \supseteq B(x, r).$$

Let $s>\udim \mu$, so that there exist $c$ and $\delta_0>0$ such that, for all $\delta<\delta_0, x\in \supp\mu$, there exists $r>0$ such that:
\begin{equation*}\label{coso}
r\in [\delta^{1/\theta}, \delta] \: ,\:  \mu(B(x, r)) \geq c r^s.    
\end{equation*}
%\begin{equation}
%\exists \: \: c, \delta_0 \: / \: \forall \delta<\delta_0, x\in \supp\mu \:, \: \exists \: r\in [\delta^{1/\theta}, \delta] \text{ with } \mu(B(x, r)) \geq c \cdot r^s.    
%\end{equation}
We want to see that the same holds for $\mu_f$. For this, we show that there exist $c'$ and $\de_0^{'}$ such that, for all $\de<\de_0^{'}\: $ and $\: y \in \supp{\mu_f}$, there exists $r>0$ such that:

\begin{equation*}\label{coso}
r\in [\delta^{1/\theta}, \delta] \: ,\:  \mu_{f} (B(y, r)) \geq c' \cdot r^s.    
\end{equation*}

Take $\de^{'}_0, \de$ and $y$ as follows: $\de_0^{'} =\lambda  \delta_0$, $\delta<\delta_0^{'}$, and $y \in \supp{\mu_f} = f(\supp\mu)$. Then, we have that $y=f(x)$ for some $x \in \supp \mu$ and so, there exists $r \in \left[\frac{\delta^{1/\theta}}{\lambda^{1/\theta}}, \frac{\delta}{\lambda} \right]$ with $\mu(B(x, r))\geq c r^s.$

Then, $\lambda r \in \left[  \frac{\lambda}{\lambda^{1/\theta}} \delta^{1/\theta}, \delta   \right]$ and also:

$$\mu_f (B(y, \lambda r)) \geq \mu( B(x, r) )\geq c \cdot r^s = \frac{c}{\lambda^s} (\lambda  r)^s,$$
so $\lambda r$ satisfies the desired inequality. However, $\lambda \cdot r$ might not fall in the ''appropriate'' range, because $\frac{\lambda}{\lambda^{1/\theta}} \leq 1$. To solve this, we proceed by cases just like before:

\begin{itemize}
    \item $\lambda r\geq \delta^{1/\theta}$: In this case, there is nothing to be done, because the same $r$ works.
    \item $\lambda r\leq \delta^{1/\theta}$:
    In this case, the ''$r$'' we are looking for will be $\delta^{1/\theta}$:

    $$\mu_f (B (\delta^{1/\theta})) \geq \mu_f (B(\lambda  r))\geq \frac{c}{\lambda^s} (\lambda r)^s\geq \frac{c}{\lambda^{s/\theta}} (\delta^{1/\theta})^s$$
\end{itemize}

So, changing the constant $c$ by $c' = \min\{\frac{c}{\lambda^{s/\theta}}, \frac{c}{\lambda^{s}}\}$, we have that $s>\udim{\mu_f}$.
\end{proof}

Applying this lemma to the projections $P_V$ for $V \in G(n, m)$, we obtain:

\begin{corollary}
$\udim{\mu_V}\leq \udim \mu$ for all $V \in G(n, m)$.
\end{corollary}

Now, in order to complete our analysis of projections we will prove that, under certain hypothesis, $\udim(\mu_V) = \udim(\mu)$ for almost all $V\in G(n, m)$. We need the following lemma: 

\begin{lemma}\label{lemadelctp}
    Let $\mu$ be a Borel measure with compact support in $\Rn$. Let $\theta\in (0,1]$, $m\in\{1,...,n-1\}$ and $t\in(0,m)$.
    
  If  $\displaystyle\limsup_{r\to 0}\frac{\log C_{r,\theta}^{t,m}(\mu)}{-\log r} - t \geq0$ then $\udim \mu_V\geq  t$ for $\gamma_{n,m}$-almost all $V\in G(n,m).$
    
\end{lemma}

\begin{proof}

Let $t'>0$ such that $\displaystyle\limsup_{r\to 0}\frac{\log C_{r,\theta}^{t',m}(\mu)}{-\log r} - t' \geq 0$ and let $t\in(0,t')$. Then, we have that $\limsup\frac{\log C_{r,\theta}^{t,m}(\mu)}{-\log r} - t> 0$.

Now, let $\varepsilon>0$ sufficiently small such that $\limsup\frac{\log C_{r,\theta}^{t,m}(\mu)}{-\log r} - t> \varepsilon$ and let $(r_k)_{k\in\mathbb{N}}$ be a sequence tending to $0$ such that $0<r_k\leq 2^{-k}$ and 
$$\lim_{k\to\infty} \frac{\log C_{r_k,\theta}^{t,m}(\mu)}{-\log r_k} = \limsup_{r\to 0} \frac{\log C_{r,\theta}^{t,m}(\mu)}{-\log r} .$$ 

Also, let $\lambda_k:=\displaystyle\inf_{x\in \supp\mu}\int \phi_{r_k,\theta}^{t,m}(x-y)d\mu(y) = \left (C_{r_k, \theta}^{t, m}\mu \right)^{-1}.$ Then there exists a subsequence $\{r_{k_{j}}\}_j$ such that $\lambda_{k_j}\leq  r_{k_j}^{t+\varepsilon}$.

Then, by lemma \ref{lemadelkernel}, for all $j$ there exists $x_j$ such that:

$$\iint  \widetilde{\phi}_{r_{k_j},\theta}^{t}(P_V x_j - P_V y)d\gamma_{n,m}(V)d\mu(y)\leq a^{-1}(1+\varepsilon)\lambda_{k_j},$$ so 
    $$\iint  \lambda_{k_j}^{-1}r_{k_j}^{\varepsilon}\widetilde{\phi}_{r_{k_j},\theta}^{t}(P_V x_j - P_V y)d\gamma_{n,m}(V)d\mu(y)\leq a^{-1}(1+\varepsilon)r_{k_j}^{\varepsilon}.$$

Therefore, summing and using Fubini's Theorem, we obtain the following:

\begin{align*}
\int \sum_{j=1}^{\infty}\int \lambda_{k_j}^{-1}r_k^{\varepsilon}\widetilde{\phi}_{r_{k_j},\theta}^{t}(P_V x_j - P_V y)d\mu(y)d\gamma_{n,m}(V)\leq a^{-1}(1+\varepsilon)\sum_{j=1}^{\infty}r_{k_j}^{\varepsilon}<\infty
\end{align*}

since $r_{k_j}^{\varepsilon}\leq 2^{-k_j \varepsilon}$. Then, for almost all $V$ there exists $M_V$ such that 
$$\int \widetilde{\phi}_{r_{k_j},\theta}^{t}(v_j - u)d\mu_V(u)=\int \widetilde{\phi}_{r_{k_j},\theta}^{t}(P_V x_j - P_V y)d\mu(y)<M_V \lambda_{k_j} r_{k_j}^{-\varepsilon}$$ where $v_j=P_Vx_{j}$. Then, $$\displaystyle\inf_{v\in \supp{\mu_V}}\int \widetilde{\phi}_{r_{k_j},\theta}^{t}(v - u)d\mu_V(u)\leq M_V r_{k_j}^{t+\varepsilon}r_{k_j}^{-\varepsilon} =M_V \cdot r_{k_j}^{t}.$$

By lemma \ref{primerlema}, we conclude that  
$$t \leq \udim \mu_V ,$$ for almost all $V$ and the result follows by letting $t\to t'$.
\end{proof}

\begin{remark}

A similar statement can be proven for the lower intermediate dimensions. For this, it is useful to remember observation \ref{tecnicismolower}.

\end{remark}

%\begin{lemma} 
%    \textcolor{red}{se puede hace sin estas hip?}Let $\mu$ be a Borel measure with compact support in $\Rn$ and let $m\in\{1,...,n-1\}$. If $\underset{V\in G(n,m)}{\sup}\dim_A \mu_V\leq m$ then
% $$\limsup_{r\to 0}\frac{\log C_{r,\theta}^{s,m}(\mu)}{-\log r}-s\geq 0.$$  for all $s\in[0,\udim \mu_V]$ and all $V\in G(n,m)$. 
%\end{lemma}

%\begin{proof}

%For all $x\in \supp \mu$ and all $V\in G(n,m)$ we have
%    \begin{align*}
%        \int \phi_{r,\theta}^{s,m}(x-y)d\mu (y)&\leq \int \phi_{r,\theta}^{s,m}(P_V x-P_V y)d\mu (y)\\
%        &=\int \phi_{r,\theta}^{s,m}(P_V x-u)d\mu_V (u)
%    \end{align*}
%and therefore
%    $$\inf_{x\in \supp \mu} \int \phi_{r,\theta}^{s,m}(x-y)d\mu (y)\leq \inf_{v\in \supp{\mu_V}}\int \phi_{r,\theta}^{s,m}(v-u)d\mu_V (u),$$ 

%which implies that
% $$\limsup_{r\to 0}\frac{\log C_{r,\theta}^{s,m}(\mu)}{-\log r}-s\geq \limsup_{r\to 0}\frac{\log C_{r,\theta}^{s,m}(\mu_V)}{-\log r}-s.$$

%Since $m>\dim_A \mu_V $ by hypothesis, using Theorem \ref{teocap} we obtain that the right-hand side of the inequality is bigger than $0$ for all $s\in[0, \udim \mu_V]$ and the result follows.

%\end{proof}

%$\udim (\mu_V) \leq \udim(\mu)$

%Combining these two results and Theorem \ref{teocap}, we obtain the following:

With these results in mind, we can state the analogue of Theorem 5.1(\cite{projections}) for measures:

\begin{theorem}\label{AAA}
Let $\mu$ be a  finite Borel measure supported on a compact subset of $\Rn$. Let $m\in\{1,...,n-1\}$ and suppose that $\dim_A \mu <m$. Then
$$\udim \mu_V\leq \left\{\text{ the unique $s\in [0,m]$ such that $\limsup_{r\to 0}\frac{\log C_{r,\theta}^{s,m}(\mu)}{-\log r}-s=0$} \right\}=\udim\mu$$ 
with an equality for $\gamma_{n,m}-$ almost all $V\in G(n,m).$ 
\end{theorem}

\begin{proof}

The theorem can be immediately deduced combining Theorem \ref{teocap} and Lemma \ref{lemadelctp}.
\end{proof}

\begin{remark}\label{AAAL}
Given that Theorem \ref{teocap} and Lemma \ref{lemadelctp} remain true for the lower intermediate dimensions, Thorem \ref{AAA} is also true.
\end{remark}

In order to motivate the main result from section \ref{section6}, we note the following:

\begin{corollary}\label{ellim}
Let \( \mu \) be a finite Borel measure with compact support in \( \mathbb{R}^n \), such that for some $m\in \{1, \dots n\}$ , $\dim_A \mu < m.$ Then, for almost every \( V \in G(n,m) \),
\[
\lim_{\theta \to 0} \underline{\dim}_\theta \mu_V
=  \lim_{\theta \to 0} \underline{\dim}_\theta \mu.
\]

\end{corollary}

\begin{remark}  
In \cite[Corollary 3.7]{NicoUrsula}, the authors prove that for a compact set \( E \subset \mathbb{R}^n \), if \( m \geq \dim_{qA} E \), then \( \dim_\theta E = \dim_\theta P_V E \) for \( \gamma_{n,m} \)-almost every \( V \in G(n,m) \), where $\dim_{qA}$ means \textit{quasi Assouad dimension} (see \cite{fraser2018assouadspectrumquasiassouaddimension} and \cite{hare2018quasidoublingselfsimilarmeasuresoverlaps}). Based on this result, an interesting question arises: can we replace \( \dim_A \) with \( \dim_{qA} \) in Theorem \ref{AAA}?  
\end{remark}

\section{Limit behavior and projections}\label{section6}

To conclude our research, we will prove a theorem that relates the Assouad dimensions of projections of a measure and the limit $\underset{\te \rightarrow 0}{\lim}\ldim \mu$ (remember that, if the dimension function $\te \rightarrow \ldim\mu$ is continuous, then this limit is simply $\dim_H^{*}\mu$). These results are the measure version of a result proven in \cite{NicoUrsula}.

\begin{lemma}\label{lem1}
Let \( r > 0 \) and \( s \leq m \). The function \( f(\theta) := C_{r,\theta}^{m,s}(\mu) \) is non-decreasing.
\end{lemma}

\begin{proof}
The result follows directly from the observation that the kernels \( \phi_{r,\theta}^{m,s}(x) \) are decreasing functions of \( \theta \).
\end{proof}

\begin{theorem}\label{theo222}
Let \( \mu \) be a finite Borel measure with compact support in \( \mathbb{R}^n \), and suppose that \( \dim_A \mu < \infty \). 
Then for almost every \( V \in G(n,m) \),
\[\min\{m,\lim_{\theta \to 0} \underline{\dim}_\theta \mu\}\leq
\lim_{\theta \to 0} \underline{\dim}_\theta \mu_V
\leq  \lim_{\theta \to 0} \underline{\dim}_\theta \mu.
\]

\end{theorem}

\begin{proof}

The inequality
\[
\lim_{\theta \to 0} \underline{\dim}_\theta \mu_V
\leq \lim_{\theta \to 0} \underline{\dim}_\theta \mu
\]
is straightforward. We proceed to prove the left-hand inequality.

Let \( t > \dim_A \mu \) and \( s < \min\left\{ \displaystyle\lim_{\theta\to 0}\underline{\dim}_\theta \mu,\, m \right\} \).  
By examining the kernel \( \phi_{r,\theta}^{s,m} \) separately on the intervals \( [0, r^\theta] \) and \( (r^\theta, \infty) \), it is not hard to see that
\begin{equation}\label{ineq1}
\phi_{r,\theta}^{s,m}(x) \leq \phi_{r,\theta}^{s,t}(x) + \mu(\mathbb{R}^n)\, r^{s(1-\theta)}.
\end{equation}

Now, since \( s < \underline{\dim}_\theta \mu \) for all $\theta$ sufficiently small, and \( \dim_A \mu < t \), we have
\[
\liminf_{r \to 0} \frac{\log C_{r,\theta}^{s,t}(\mu)}{-\log r} - s > 0,
\]

for that $\theta$. Therefore, for each such $\theta$ there exist \( r_0 > 0 \) and a constant \( c > 1 \) such that, for all \( r \in (0, r_0) \), there exists a point \( x_r \in \operatorname{supp}(\mu) \) satisfying
\begin{equation}\label{ineq2}
    \int \phi_{r,\theta}^{s,t}(x_r - y) \, d\mu(y) \leq c\, r^s.
\end{equation}

Combining inequalities \eqref{ineq1} and \eqref{ineq2}, we obtain
\[
C_{r,\theta}^{s,m}(\mu) \leq C\, r^{s(1-\theta)},
\]
for some constant \( C > 0 \), and thus
\[
\lim_{\theta \to 0} \liminf_{r \to 0} \left( \frac{\log C_{r,\theta}^{s,m}(\mu)}{-\log r} - s \right) \geq 0.
\]

Moreover, since, by Lemma~\ref{lem1}, the function
\[
\theta \mapsto \liminf_{r \to 0} \left( \frac{\log C_{r,\theta}^{s,m}(\mu)}{-\log r} - s \right)
\]
is non-decreasing, we conclude that
\[
\liminf_{r \to 0} \left( \frac{\log C_{r,\theta}^{s,m}(\mu)}{-\log r} - s \right) \geq 0,
\]
for all \( \theta \) sufficiently small. Therefore, for each such \( \theta \), we have \( \underline{\dim}_\theta \mu_V \geq s \) for almost every \( V \in G(n,m) \), by Lemma \ref{lemadelctp}.

 Now let \( \{ \theta_n \} \) be a decreasing sequence tending to zero, and define
\[
V_i := \left\{ V \in G(n,m) : \underline{\dim}_{\theta_i} \mu_V \geq s \right\},
\quad \text{and} \quad A_k := \bigcap_{i=1}^{k} V_i.
\]
Then \( \gamma_{n,m}\left( \bigcap_{k=1}^\infty A_k \right) = 1 \), and for all \( V \in \bigcap_{k=1}^\infty A_k \), we have
\[
s \leq \lim_{\theta \to 0} \underline{\dim}_\theta \mu_V.
\]
The result follows by letting \( s \to \min\left\{ m,\, \lim_{\theta \to 0} \underline{\dim}_\theta \mu \right\} \).

\end{proof}
\begin{remark}
The same result the upper intermediate dimension is proven similarly.
\end{remark}

By the proof of Theorem~\ref{theo222}, we obtain the interesting conclusion that the Assouad dimensions of the pushforward measures are closely related to the behavior of the $\theta$-intermediate dimensions of the original measure. 

\begin{corollary} Let $\mu$ be a finite Borel measure with compact support in $\mathbb{R}^n$, and suppose that $\dim_A \mu<\infty$.
    Then 
    \begin{itemize}
        \item If $\overline{\dim}_B \mu_V=\dim_A \mu_V$
    for $\gamma_{n,m}$ almost all $V\in G(n,m)$, then $$\lim_{\theta\to 0}\overline{\dim}_\theta \mu\geq \dim_A \mu_V$$ for $\gamma_{n,m}$ almost all $V\in G(n,m)$.\\
    \item If $\lim_{\theta\to 0}\overline{\dim}_\theta \mu\geq \dim_A \mu_V$ for $\gamma_{n,m}$ almost all $V\in G(n,m)$ then 
    $$\min\{m,\dim_A \mu_V\}\leq\overline{\dim}_B \mu_V\leq\dim_A \mu_V$$
    for $\gamma_{n,m}$ almost all $V\in G(n,m)$.
    \end{itemize}

\end{corollary} \begin{proof}
    The corollary follows from theorem \ref{theo222} and the fact that if $\overline{\dim}_B \mu_V=\dim_A \mu_V$ then $\overline{\dim}_\theta \mu_V =\dim_A\mu_V$ for all $\theta$ (see Corollary \ref{cordegeq}).
\end{proof}

%\begin{theorem}
%Let $\mu$ be a  finite Borel measure supported on a compact subset of $\Rn$. Let $m\in\{1,...,n-1\}$ and suppose that $\dim_A \mu_V <m$ for all $V\in G(n,m)$. then

%$$\udim \mu_V\leq \left\{\text{ the unique $s\in [0,m]$ such that $\limsup_{r\to 0}\frac{\log C_{r,\theta}^{s,m}(\mu)}{-\log r}-s=0$} \right\}=\udim\mu$$ with an equality for $\gamma_{n,m}-$ almost all $V\in G(n,m).$ 
%\end{theorem}

%aca aparecen la bibliografia
\noindent
{\bf Acknowledgments}:
We thank the anonymous referee for their valuable comments and suggestions, which helped improve the final version.

The research of the authors is partially supported by Grants
PICT 2022-4875 (ANPCyT),
PIP 202287/22 (CONICET), and
UBACyT 2022-154 (UBA). Nicol\'as Angelini is also partially supported by PROICO 3-0720 ``An\'alisis Real y Funcional. Ec. Diferenciales''.

\bibliographystyle{plain}
\bibliography{ref}

\end{document}